\pdfoutput=1
\newif\ifarxiv
\newif\ifau
\arxivtrue

\ifarxiv
\documentclass{amsart}
\usepackage[utf8]{inputenc}
\usepackage[letterpaper,left=1in,top=1in,right=1in,bottom=1in]{geometry}
\usepackage[numbers,square]{natbib}
\usepackage{graphicx}
\fi

\usepackage{amsmath,amssymb,url}
\usepackage{xcolor}
\usepackage{yhmath}
\usepackage{hyperref}
\usepackage{enumitem} 
\numberwithin{equation}{section}

\theoremstyle{plain}
\newtheorem{theorem}{Theorem}[section]
\newtheorem{lemma}[theorem]{Lemma}

\newtheorem{corollary}[theorem]{Corollary}

\theoremstyle{definition}
\newtheorem{definition}[theorem]{Definition}

\newtheorem{remark}[theorem]{Remark}
\newtheorem*{note}{Note} 

\makeatletter
\newenvironment{cleanproof}{%
  \par
  \pushQED{\qed}%
  \normalfont \topsep6\p@\@plus6\p@\relax
  \trivlist
  \item\relax
}{%
  \popQED\endtrivlist\@endpefalse
}
\makeatother

\def\set#1{#1}
\def\relation#1{#1}
\def\clone#1{\mathcal{#1}}
\def\minion#1{\mathcal{#1}}
\def\Bc{\clone{B}}
\def\Ac{\clone{A}}
\def\Lc{\clone{L}}
\def\Cc{\clone{C}}
\def\Mc{\clone{M}}
\def\Ic{\clone{I}}
\def\Pc{\Pi}

\def\atleast#1{\uparrow^{#1}}

\newdimen\xxxwidth
\setbox0=\hbox{$xx\cdots xxx$} \xxxwidth=\wd0
\def\overbounded#1{\overbrace{\hbox to\xxxwidth{\hss$#1$\hss}}}

\def\za#1{[#1]}

\def\C#1{{\mathtt{P}}(#1)}

\def\Pol{{\rm Pol}}
\def\PolId{{\rm PolId}}
\def\UP{{\rm UP}}
\def\adef#1{\catcode`#1=13 \begingroup \lccode`\~=`#1\lowercase{\endgroup\def~}}
\def\colorxy{\adef x{{\textcolor{red}{x}}}\adef y{{\textcolor{blue}{y}}}}
\def\sm#1{{\small{#1}}}
\def\smatrix#1{\setbox0=\hbox{$\displaystyle\begin{matrix}#1\end{matrix}$}\ht0=\dimexpr\ht0-8pt\relax%
\mathop{\left\{\vcenter{\box0}\right\}}^{\phantom{j}}}

\def\leb#1{\setbox0=\hbox{\vrule depth 1mm width 0pt\tiny\ref{#1}}%
\raise2mm\vbox{\kern-2mm\moveleft2mm%
\hbox{\rlap{\color{white}\vrule height\dimexpr\ht0+1mm depth\dp0 width\wd0}\box0}}%
}

\newlength{\miniaturewidth}
\ifau
\fi
\ifarxiv
\fi

\begin{document}

\title[Products of Boolean Clones]{Products of Boolean Clones up to Minion Homomorphisms}

\ifarxiv
\author[R. Olšák]{Radek Olšák}
\fi
\address{Department of Algebra\\Faculty of Mathematics and Physics\\
Charles University\\Prague\\Czech Republic}
\urladdr{https://radek.olsak.net}
\email{radek@olsak.net}


\thanks{
Funded by the European Union (ERC, POCOCOP, 101071674). Views and opinions expressed are however those of the author(s) only and do not necessarily reflect those of the European Union or the European Research Council Executive Agency. Neither the European Union nor the granting authority can be held responsible for them.
}


\begin{abstract}
We study the minion homomorphism poset of multisorted Boolean clones. Our main result completely determines the subposet induced by products of Boolean clones: every inequality between products is induced by inequalities between their Boolean factors. Furthermore, we determine the minion cores of all products of Boolean clones, thereby identifying canonical representatives of their minion homomorphic equivalence classes.
\end{abstract}

\maketitle

\section{Introduction}\label{sec:intro}

    In universal algebra \cite{universal1, universal2, universal3}, obtaining a complete characterization of clones on finite domains remains a long-standing goal. While Post \cite{post} successfully classified all clones on a two-element domain, the problem becomes notoriously difficult for larger domains. 

  Clones on finite domains can equivalently be described as sets of functions preserving a given collection of relations. The characterization problem is then equivalent~\cite{bodnarchuk,geiger} to characterizing relational structures up to primitive positive definability (see \cite{cspsurvey}).  
   This relational perspective links clones to the computational complexity of constraint satisfaction problems (CSPs) \cite{cspsurvey}. Motivated by this connection, \cite{wonderland, approach} proposes a coarser classification program: the classification of clones up to minion homomorphisms.
   
 Minion homomorphisms admit two equivalent characterizations. From the functional perspective, they are defined as mappings between function clones preserving composition with projections. They can equivalently be defined as mappings preserving height-one identities, i.e., identities with exactly one occurrence of a function symbol on each side.
 From the relational perspective, minion homomorphisms correspond to  primitive positive constructability \cite{wonderland}. 
 
 The existence of a minion homomorphism defines a preorder on the class of clones.
 The central object of this paper, the \emph{minion homomorphism poset}, is obtained by identifying homomorphically equivalent clones.
       Several parts of this poset have already been characterized \cite{vucaj, florian1, florian2, florian3, kapytka}. Of particular relevance to this paper is the result in \cite{vucaj}, which uses Post's classification of Boolean clones \cite{post} to determine the complete poset of Boolean clones modulo minion homomorphisms, see Figure~\ref{palanda}.         

    Our result generalizes this characterization to a multisorted setting. Multisorted clones generalize standard clones by operating over an indexed family of domains, called sorts. In this framework, an $n$-ary member of the clone is given by a tuple of functions, one for each sort, where each component takes $n$ arguments from its respective sort and returns a value of the same sort. Relations are likewise sorted, with each coordinate assigned a sort.

    Multisorted clones arise naturally in the clone theory \cite{romov} and in the study of multisorted CSPs, where variables are restricted to specific sorts \cite{david}. They also provide a way to study (single-sorted) clones over larger domains through their action on smaller domains \cite{kapytka}.  
    
    Although a complete classification of multisorted Boolean clones remains a formidable task, even up to minion homomorphic equivalence, it may ultimately prove tractable. This possibility is supported by a result of Taimanov \cite{taimanov1, taimanov2, taimanov3}, which establishes that the number of multisorted Boolean clones is at most countable (see also \cite{david} for recent progress). 

    The first main result of this paper describes the minion homomorphism poset restricted to a subclass of multisorted Boolean clones, namely, \emph{products of Boolean clones}. The $n$-ary members of the product of clones $\clone{C}_1, \dots\clone{C}_m $ are \emph{all} tuples $(f_1,\dots,f_m)$ of $n$-ary functions with $f_i \in \clone{C}_i$ for each $i$; that is, there are no constraints relating the $f_i$. From the relational perspective, these are multisorted clones that can be determined by relations whose coordinates all have the same sort. In this sense, our result complements that of \cite{kapytka}, where coordinates were allowed to have different sorts, but only relations of arity at most two were considered. 

In the minion homomorphism poset, the product of clones coincides with their meet (i.e., infimum). Thus, the above result may be viewed as taking the class of Boolean clones and closing it under meets.
    
    Some inequalities among products of clones in the minion homomorphism poset already follow from the inequalities between Boolean clones. For example, if  $\clone{A} \le \clone{U}$, $\clone{B} \le \clone{U}$, and $\clone{C} \le \clone{V}$ for some Boolean clones, then $\clone{A} \wedge \clone{B} \wedge \clone{C} \le \clone{U} \wedge \clone{V}$; therefore the product of $\clone{A}$, $\clone{B}$, and $\clone{C}$ is less than or equal to the product of $\clone{U}$ and $\clone{V}$. We prove that, in fact, these are the only inequalities in this class.

An interesting consequence for Boolean clones is the following. 
Suppose that $\clone{B}$ is a Boolean clone and that $A$ is a family of Boolean clones such that no clone in $A$ is less than or equal to $\clone{B}$ (with respect to the order in Figure~\ref{palanda}). Then there exists a set of height-one identities satisfiable in every clone in $A$ and not satisfiable in $\clone{B}$. (Indeed, the assumption implies that the product of the clones in $A$ is not below $\clone{B}$, and then the proof provides height-one identities distinguishing them.)

Recall that
the minion homomorphism order is obtained from the minion homomorphism preorder by identifying homomorphically equivalent clones.
It turns out that each equivalence class contains a natural and canonical representative, provided we work in the more general setting of minions rather than clones. A minion is defined as a family of functions $A \to B$ that is closed under composition with $A$-projections. It follows from \cite{kapytka} that every minion on finite sets (in particular, every multisorted Boolean clone) is equivalent to a unique (up to isomorphism) minion core: a minion in which every endomorphism is an automorphism. 
Our second main result, presented in the final section, is an explicit description of the minion core of each product of Boolean clones.

\begin{figure}
\centering{
{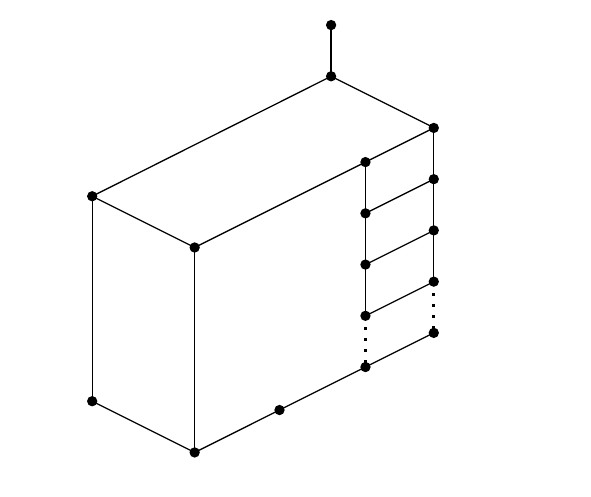}
\caption{Lattice of Boolean clones modulo minion homomorphisms}
\label{palanda}
}
\end{figure}

\section{Multisorted clones}

\begin{definition}[$\za{n}$] 
By $[n]$ we denote the set $\{1,2,\dots,n\}$. 
\end{definition}

\subsection{Function clones}

\begin{definition}[function composition] Let $\set{A}$ be a set. Let $f$ be a function $\set{A}^n\to \set{A}$ and $g_1,\dots,g_n$ be functions $\set{A}^m\to \set{A}$. The {\em composition} of $f$ with $g_1,\dots,g_n$, denoted by $f\circ (g_1,\dots,g_n)$, is the $m$-ary function defined by
$$\big(f\circ (g_1,\dots,g_n)\big)(x_1,\dots, x_m) = f(g_1(x_1,\dots,x_m),\dots,g_n(x_1,\dots,x_m))$$
\end{definition}

\begin{definition}[projections]
Let $\set{A}$ be a set. The $n$-ary {\em projection} onto the $i$-th coordinate on $\set{A}$ is the function $\pi_i^n \colon \set{A}^n\to \set{A}$ defined by
$$\pi_i^n(x_1,\dots,x_n) = x_i$$
\end{definition}

\begin{definition}[function clone]
Let $\set{A}$ be a set. A set $\clone{C}$ of functions $\set{A}^n\to \set{A}$ (for various $n$) is called a {\em function clone} if it contains all projections of all arities and is closed under composition.
\end{definition}

The term \emph{clone} will always refer to a function clone.

\subsection{Relations}

We now define compatibility of functions with relations, which provides an alternative characterization of clones.

\begin{definition}[compatibility with a relation]
  Let $\relation{R}$ be an $m$-ary relation on a set $\set{A}$ and let $f\colon \set{A}^n \to \set{A}$ be a function. We say that $f$ is {\em compatible} with $\relation{R}$ if for all $a_{i,j}\in \set{A}$ ($1\leq i\leq m, 1\leq j\leq n$) we have that $\forall j\colon(a_{1,j},\dots,a_{m,j})\in \relation{R}$ implies $(f(a_{1,1},\dots,a_{1,n}),\dots,f(a_{m,1},\dots,a_{m,n}))\in \relation{R}$.
\end{definition}

This notion is illustrated by the following diagram.

$$\begin{matrix}
f(& a_{1,1} & a_{1,2} & \dots & a_{1,n} & ) & = & b_1 \cr
f(& a_{2,1} & a_{2,2} & \dots & a_{2,n} & ) & = & b_2 \cr
& \vdots & \vdots & \ddots & \vdots &  & \vdots & \vdots \cr
f(& a_{m,1} & a_{m,2} & \dots & a_{m,n} & ) & = & b_m \cr
& \in & \in & \dots & \in &  &  & \in \cr
& R & R & \dots & R &  & \Rightarrow & R \cr
\end{matrix}$$

\begin{definition}[compatibility with set of relations]
Let $f$ be a function on a set $\set{A}$ and let $\Gamma$ be a set of relations on $\set{A}$. We say that $f$ is {\it compatible} with $\Gamma$ if it is compatible with all relations $\relation{R}\in\Gamma$.
\end{definition}

\begin{definition}[Pol]
For a set of relations $\Gamma$, we denote by $\Pol(\Gamma)$ the set of all functions compatible with every relation in $\Gamma$.
\end{definition}

\begin{theorem}
\cite{bergman} Let $\Gamma$ be a set of relations on a set $\set{A}$. Then $\Pol(\Gamma)$ is a clone.
\end{theorem}

It is well known \cite{bodnarchuk,geiger} that every clone on a finite set can be described as $\Pol(\Gamma)$ for some set of relations $\Gamma$.  

\subsection{Multisorted clones}

\begin{definition}[$k$-sorted set]
A {\it $k$-sorted set} $\set{A}=(\set{A}_1,\set{A}_2,\dots,\set{A}_k)$ is a $k$-tuple of sets.
\end{definition}

\begin{definition}[sort]
Given a $k$-sorted set $\set{A}$, we refer to an index $i\in\{1,\dots,k\}$ as a \emph{sort}.
\end{definition}

\begin{definition}[$k$-sorted function]
Let $\set{A}=(\set{A}_1,\set{A}_2,\dots,\set{A}_k)$ and $\set{B}=(\set{B}_1,\set{B}_2,\dots,\set{B}_k)$ be $k$-sorted sets. A {\em $k$-sorted function} from $\set{A}^n \to \set{B}$ is a tuple $f=(f_1,f_2,\dots,f_k)$, where, for each $1\leq i\leq k$, $f_i$ is a function $\set{A}_i^n\to \set{B}_i$.
\end{definition}

\begin{definition}[$k$-sorted composition]
Let $\set{A}$ be a $k$-sorted set. Let $f=(f_1,\dots,f_k)$ be a $k$-sorted function $\set{A}^n\to \set{A}$, and let $g_1=(g_{1,1},\dots,g_{1,k}),\dots,g_n=(g_{n,1},\dots,g_{n,k})$ be $k$-sorted functions of arity $m$. The {\em composition} $f\circ(g_1,\dots g_n)$ is defined componentwise by
$$f\circ(g_1,\dots g_n) = (f_1\circ (g_{1,1},\dots,g_{n,1}), \dots, f_k\circ (g_{1,k},\dots,g_{n,k}))$$ 
\end{definition}

\begin{definition}[$k$-sorted projection]
Let $\set{A}$ be a $k$-sorted set. The {\em $k$-sorted projection} of arity $n$ onto the $i$-th variable is the $k$-sorted function $\pi_i^n = (\pi_i^n,\dots ,\pi_i^n)$
\end{definition}

\begin{definition}[$k$-sorted clone]
Let $\set{A}$ be a $k$-sorted set. A set $\clone{C}$ of $k$-sorted functions on $\set{A}$ is called a {\it $k$-sorted clone} if it contains all the $k$-sorted projections of all arities and is closed under $k$-sorted composition.
\end{definition}

An alternative way to define a $k$-sorted clone is as a special type of single-sorted clone on the product $\prod \set{A}_i$, whose functions act independently on the individual sorts. More precisely, a $k$-sorted function $f=(f_1,f_2,\dots,f_k)$ may be viewed as a function $\hat{f}:\prod ({\mathcal A}_i)^n \to \prod {\mathcal B}_i$ defined by
\newcommand{\vmapsto}{\rotatebox[origin=c]{-90}{$\mapsto$}}
\begin{align*}
  &\left( (a^1_1, \dots, a^n_1), (a^1_2,\dots, a_2^n), \dots, (a_k^1, \dots, a_k^n)\right)\cr
  &\hskip3.5cm\vmapsto \cr
  &\left( f_1(a^1_1, \dots, a^n_1), f_2(a^1_2,\dots, a_2^n), \dots, f_k(a_k^1, \dots, a_k^n)\right) .
\end{align*}
 
Using this representation, we can define compositions and projections for multisorted functions exactly as in the single-sorted case. These definitions coincide with those given above.

\subsection{Multisorted relations}

Unlike multisorted functions, whose components act independently on each sort, multisorted relations can express interactions between different sorts.

\begin{definition}[$k$-sorted relation]
Let $\set{A}=(\set{A}_1,\dots,\set{A}_k)$ be a $k$-sorted set, and let $i_1,\dots, i_m \in [k]$. A {\em $k$-sorted relation} of arity $m$ with type $(i_1,\dots, i_m )$ is a subset $\relation{R}\subseteq \set{A}_{i_1}\times \set{A}_{i_2}\times\cdots\times \set{A}_{i_m}$.
\end{definition}

Note that the type $(i_1,\dots,i_m)$ is part of the definition of a multisorted relation. In particular, a multisorted relation cannot be specified without its type. This distinction remains important even when all sets $\set{A}_i$ are equal.

\begin{definition}[compatibility with multisorted relation]
Let $\set{A}=(\set{A}_1,\dots,\set{A}_k)$, a $k$-sorted set, let $f=(f_1,\dots,f_k)$ a $k$-sorted function $\set{A}^n\to \set{A}$, and let $\relation{R}$ a $k$-sorted relation of arity $m$ with type $(\ell_1,\dots, \ell_m)$. We say that $f$ is {\em compatible} with $\relation{R}$ if for all $a_{i,j} \in \set{A}_{\ell_i}$ ($1\leq i\leq m$, $1\leq j\leq n$) it holds that $\forall j\colon (a_{1,j},\dots,a_{m,j})\in \relation{R}$ implies that
$$(f_{\ell_1} (a_{1,1},\dots,a_{1,n}),\dots,f_{\ell_m}(a_{m,1},\dots,a_{m,n}))\in \relation{R}.$$
\end{definition}

This notion is illustrated by the following diagram.

$$\begin{matrix}
f_{\ell_1}(& a_{1,1} & a_{1,2} & \dots & a_{1,n} & ) & = & b_1 \cr
f_{\ell_2}(& a_{2,1} & a_{2,2} & \dots & a_{2,n} & ) & = & b_2 \cr
& \vdots & \vdots & \ddots & \vdots &  & \vdots & \vdots \cr
f_{\ell_m}(& a_{m,1} & a_{m,2} & \dots & a_{m,n} & ) & = & b_m \cr
& \in & \in & \dots & \in &  &  & \in \cr
& R & R & \dots & R &  & \Rightarrow & R \cr
\end{matrix}$$

In this paper, we focus on the special case of multisorted relations whose type consists of a single sort. That is, all indices in the type are equal. In the definition above, this means that all functions $f_{\ell_i}$ are identical. Consequently, compatibility with such relations amounts to imposing constraints on each sort independently.

\begin{definition}[simple $k$-sorted relation]
A {\em simple $k$-sorted relation} is $k$-sorted relation of type $(i,i,\dots,i)$ for some $i$.
\end{definition}

We define the clone product only for finitely many factors. As we shall see, allowing infinite products would not make a difference in our context.

\begin{definition}[clone product]
Let $\clone{C}_1,\clone{C}_2,\dots,\clone{C}_k$ be clones on a set $\set{A}$. The {\em product} of $\clone{C}_1,\clone{C}_2,\dots,\clone{C}_k$, denoted by $\C{\clone{C}_1,\clone{C}_2,\dots,\clone{C}_k}$ is the $k$-sorted clone consisting of all $k$-sorted functions $f = (f_1,\dots,f_k)$, where $f_i\in \clone{C}_i$.
\end{definition}

Directly from the definition of compatibility and a simple $k$-sorted relation we obtain the following lemma. 

\begin{lemma}
  \label{all-clones-are-products}
  Let $\Gamma$ be a set of simple $k$-sorted relations. For each $i\in[k]$, let $\Gamma_i$ denote the subset of $\Gamma$ containing all relations of type $(i,i,\dots,i)$. Let $\clone{C}$ be the $k$-sorted clone given by $\clone{C}=\Pol(\Gamma)$. Then $\clone{C}=\C{\Pol({\Gamma}_1),\dots,\Pol({\Gamma}_k)}$.
\end{lemma}

\subsection{Homomorphisms}

\begin{definition}[minor]
Let $\set{A}$ be a set, let $f\colon \set{A}^n\to \set{A}$ be a function, and let $i_1,i_2,\dots,i_n\in [m]$. The function $f\circ (\pi_{i_1}^m,\pi_{i_2}^m,\dots,\pi_{i_n}^m)$ is called a {\em minor} of $f$.
\end{definition}

Minors can be equivalently defined by merging and permuting variables (and introducing dummy variables). Indeed, for example, the function $g$ defined from $f$ by $\forall x,y: g(x,y) = f(y,y,x)$ is the same as $g = f \circ (\pi_2^2,\pi_2^2,\pi_1^2)$.

\begin{definition}[minion homomorphism]
Let $\clone{A}$ and $\clone{B}$ be clones. A mapping $\zeta\colon\clone{A}\to \clone{B}$ is called a {\em minion homomorphism} if it satisfies the following conditions:
\begin{itemize}
\item {\em Preservation the arities}: 
a function of arity $n$ in $\clone{A}$ gets mapped to a function of arity $n$ in $\clone{B}$.  
\item {\em Preservation of minors}: for any $n$-ary $f\in \clone{A}$  and any $i_1,i_2,\dots,i_n\in [m]$, we have 
$$\zeta(f\circ(\pi_{i_1}^m,\dots,\pi_{i_n}^m)) = \zeta(f)\circ(\pi_{i_1}^m,\dots,\pi_{i_n}^m) $$
\end{itemize}
\end{definition}

This definition originates from the more general notion of a minion \cite{approach}, which is a set of functions between possibly different sets. Since a minion does not come equipped with distinguished projections in the same way as a clone, preservation of projections is not part of the general definition of a minion homomorphism. However, in our setting, the following lemma shows that we may equivalently assume that minion homomorphisms preserve projections.

\begin{definition}[idempotent function]
Let $\set{A}$ be a set and let $f\colon \set{A}^n\to \set{A}$ be a function. We say that $f$ is {\em idempotent} if $f(x,x,\dots,x) = x$ for all $x\in\set{A}$.
\end{definition}

\begin{definition}[idempotent clone]
A clone $\clone{A}$ is called {\em idempotent} if every function in $ \clone{A}$ is idempotent.
\end{definition}

\begin{lemma}
\label{projtoproj}
Let $\clone{A}$ and $\clone{B}$ be idempotent clones, and let $\zeta:\clone{A}\to \clone{B}$ be a minion homomorphism. Then $\zeta(\pi_i^n) = \pi_i^n$ for every projection $\pi_i^n$.
\end{lemma}
\begin{proof}
Since $\clone{A}$ and $\clone{B}$ are idempotent, they contain only one unary function, the identity function $\tau(x) = x$. Thus, it has to be mapped to itself by $\zeta$. Note that $\pi_i^n = \tau\circ(\pi_i^n)$ and thus $\zeta(\pi_i^n) = \zeta(\tau)\circ(\pi_i^n) = \tau\circ(\pi_i^n) = \pi_i^n$. 
\end{proof}

\subsection{Preorder}

Minion homomorphisms induce a preorder on the class of all clones.

\begin{definition}[$\leq$]
Let $\clone{A}$ and $\clone{B}$ be clones. We write $\clone{A}\leq \clone{B}$ if there exists a minion homomorphism from $\clone{A}$ to $ \clone{B}$.
\end{definition}

\begin{lemma}
The relation $\leq$ is a preorder.
\end{lemma}

The preorder $\leq$ naturally induces an equivalence relation on clones.

\begin{definition}[$\sim$]
Let $\clone{A}$ and $\clone{B}$ be clones. We say $\clone{A}$ and $\clone{B}$ are {\it equivalent}, denoted by $\clone{A}\sim \clone{B}$ if $\clone{A}\leq \clone{B}$ and $\clone{B}\leq \clone{A}$.
\end{definition}

The minion homomorphism poset is obtained from the preoder in the standard way.

\begin{definition}[minion homomorphism poset]
The \emph{minion homomorphism poset} is the partially ordered set whose elements are $\sim$-equivalence classes of clones, with the order given by $\le$.    
\end{definition}

The following definition and lemma provide a strategy for showing that no minion homomorphism exists from one clone to another.

\begin{definition}
A {\em height-one identity} is a formal expression of the form

$$f_1(v_1,\dots,v_n)=f_2(w_1,\dots,w_m),$$

where $f_1$ and $f_2$ are function symbols (not necessarily distinct),
and $v_1,\dots,v_n,w_1,\dots,w_m$ are variables (not necessarily distinct).

Let $S$ be a set of height-one identities and let $\clone{A}$ be a clone.
We say that $\clone{A}$ {\em satisfies} $S$, and write
$\clone{A}\models S$, if the function symbols occurring in $S$ can be
interpreted by functions from $\clone{A}$ of the corresponding arities
so that all identities in $S$ hold universally.
\end{definition}

For example,
$f(x,x,y)=f(z,x,x)$
is a height-one identity. A clone $\clone{A}$ satisfies this identity if it
contains a ternary function $f$ such that
$\forall x,y,z\colon\ f(x,x,y)=f(z,x,x)$.
Observe that this condition can be expressed entirely in terms of minors:
$f(\pi_1,\pi_1,\pi_2)=f(\pi_3,\pi_1,\pi_1)$.
Since minion homomorphisms preserve minors, they also preserve the satisfaction
of height-one identities, and more generally of sets of height-one
identities. Consequently, we obtain the following lemma.

\begin{lemma}\label{identity-proof-strategy}
Let $S$ be a set of height-one identities, and let $\clone{A}$ and
$\clone{B}$ be clones. If
$\clone{A}\models S$
and
$\clone{B}\not\models S$,
then
$\clone{A}\not\leq \clone{B}$.
\end{lemma}

\section{Classification of products of Boolean clones}

In this section, we determine the subposet of the minion homomorphism poset that is
induced by products of Boolean clones. 
We begin by introducing a simpler order, $\preceq$, based on the inequalities among Boolean clones depicted in Figure \ref{palanda}. We then show that $\preceq$ agrees with $\leq$ for certain pairs of products. Finally, we prove that $\preceq$ and $\leq$ coincide on all pairs of products.

\subsection{Product as a meet}

\begin{lemma}
The product of clones $\C{\clone{X}_1,\clone{X}_2,\dots, \clone{X}_n}$ is the meet $\bigwedge \{\clone{X}_1,\clone{X}_2,\dots,\clone{X}_n\}$ in the minion homomorphism poset. 
\end{lemma}
\begin{proof}
Let $\clone{P} = \C{\clone{X}_1,\clone{X}_2,\dots, \clone{X}_n}$.

\begin{itemize}
\item For every $i$ we have $\clone{P}\leq \clone{X}_i$, since there is a minion homomorphism $\zeta: \clone{P}\to \clone{X}_i$ defined as the $i$-th projection.
\item Suppose that $\clone{M}$ is a clone such that $\forall i:\clone{M} \leq \clone{X}_i$, then $\clone{M}\leq \clone{P}$. Indeed, let $\zeta_i:\clone{M}\to\clone{X}_i$ be a minion homomorphism witnessing $\clone{M}\leq \clone{X}_i$. Then the mapping $\zeta(f) = (\zeta_1(f), \zeta_2(f), \dots, \zeta_n(f)$ defines a minion homomorphism $\zeta:\clone{M}\to\clone{P}$.
\end{itemize}
This gives us that $\clone{P}$ is really the meet $\bigwedge \{\clone{X}_1,\clone{X}_2,\dots,\clone{X}_n\}$.
\end{proof}

The following lemmas record several simple but important properties of the meet/product that will be used later.

\begin{lemma}
\label{perm}
  Let $\C{\clone{X}_1,\clone{X}_2,\dots, \clone{X}_n}$ be a product of clones. Let $p$ be a permutation of $[n]$. Then $\C{\clone{X}_1,\clone{X}_2,\dots, \clone{X}_n}\sim\C{\clone{X}_{p(1)},\clone{X}_{p(2)},\dots,\clone{X}_{p(n)}}$.
\end{lemma}

\begin{lemma}
  \label{ineq-triv}
  Let $\C{\clone{X}_1,\clone{X}_2,\dots,\clone{X}_n}$ and $\C{\clone{Y}_1,\clone{Y}_2,\dots,\clone{Y}_m}$ be products of clones. Suppose that there exist indices $j_1,j_2,\dots,j_m\in[n]$ such that $\clone{X}_{j_i}\leq \clone{Y}_i$ for every $i\in[m]$. Then $\C{\clone{X}_1,\clone{X}_2,\dots,\clone{X}_n}\leq \C{\clone{Y}_1,\clone{Y}_2,\dots,\clone{Y}_m}$.
\end{lemma}

\begin{lemma}
  \label{incomparable-prod}
  Every finite product $\C{\clone{X}_1,\clone{X}_2,\dots,\clone{X}_n}$ of clones is equivalent to a product of the form $\C{\clone{Y}_1,\clone{Y}_2,\dots,\clone{Y}_m}$, where the clones $\clone{Y}_1, \clone{Y}_2, \dots, \clone{Y}_m$ are pairwise incomparable. Moreover for every $i\in[m]$ there exist $j\in[n]$ such that $Y_i\sim X_j$.
\end{lemma}

These lemmas motivate the following definition of the ``trivial'' order on products.

\begin{definition}[$\preceq$]
Let $\C{\clone{X}_1,\clone{X}_2,\dots,\clone{X}_n}, \C{\clone{Y}_1,\clone{Y}_2,\dots,\clone{Y}_m}$ be two products of clones. We write $\C{\clone{X}_1,\clone{X}_2,\dots,\clone{X}_n}\preceq \C{\clone{Y}_1,\clone{Y}_2,\dots,\clone{Y}_m}$ if there exist indices $j_1,j_2,\dots,j_m \in [n]$ such that $\clone{X}_{j_i}\leq \clone{Y}_i$ for every $i\in[m]$.
\end{definition}

\begin{definition}[trivial equivalence]
Let $\C{\clone{X}_1,\clone{X}_2,\dots,\clone{X}_n}, \C{\clone{Y}_1,\clone{Y}_2,\dots,\clone{Y}_m}$ be two products of clones. We say that they are {\em trivially equivalent} if
\begin{align}
\C{\clone{X}_1,\clone{X}_2,\dots,\clone{X}_n}&\preceq \C{\clone{Y}_1,\clone{Y}_2,\dots,\clone{Y}_m}\cr
\C{\clone{Y}_1,\clone{Y}_2,\dots,\clone{Y}_m}&\preceq \C{\clone{X}_1,\clone{X}_2,\dots,\clone{X}_n}.
\end{align}

We denote this by $\C{\clone{X}_1,\clone{X}_2,\dots,\clone{X}_n}\approx \C{\clone{Y}_1,\clone{Y}_2,\dots,\clone{Y}_m}$
\end{definition}

Lemma~\ref{ineq-triv} implies that any two trivially equivalent products are equivalent with respect to minion homomorphisms.

\subsection{Notation for clones on 2-element domain}

From now on, we work over the two-element domain $\{0,1\}$ equipped with the natural order $\leq$, addition modulo $2$, and negation (defined by $\overline{x}=1-x$). 

Note that, apart from the top element of the poset, all clones considered in this paper are idempotent. This follows from the following observations.
\begin{itemize}
  \item If a clone $\clone{Z}$ contains a constant function $f$, then for any clone $\clone{X}$, the mapping that sends every function in $\clone{Z}$ to $f$ is a minion homomorphism. Consequently, $\clone{Z}$ is (a representative of) the top element in the minion homomorphism poset.
  \item Let $\clone{Z}$ be a Boolean clone that does not contain a constant function. Define a mapping $\zeta:\clone{Z}\to\clone{Z}$  by $\zeta(f) = f_1\circ f,$ where $f_1$ denotes the unary minor of $f$ (which is either the identity or the negation). Then $\zeta$ is an endomorphism whose image is idempotent. Therefore, $\clone{Z}$ is equivalent to an idempotent clone.  
\end{itemize}

\begin{definition}[$\PolId$] Let $R$ be a set of relations. We define $\PolId(R)$ to be the set of all idempotent functions in $\Pol(R)$. \end{definition}

\begin{definition}[Notation for specific clones]
We introduce the notation for the clones and relations used throughout the following sections.
\begin{itemize}
\item ($\relation{R}_n$) For $n\geq 2$, let $\relation{R}_n$ denote the $n$-ary relation $\relation{R}_n = \{0,1\}^n\setminus \{(1,1,\dots,1)\}.$
\item ($\Bc_n$) For $n\geq 2$, let $\Bc_n =\PolId(\relation{R}_n)$.
\item ($\Bc_{\infty}$) Let $\Bc_{\infty} = \PolId(\relation{R}_{3},\relation{R}_{4},\dots)$.
\item ($\Bc_n^{\leq}$) For $i\geq 2$ let $\Bc_n^{\leq} =\PolId(\relation{R}_n,\leq)$.
\item ($\Bc_{\infty}^{\leq}$) Let $\Bc_{\infty}^{\leq} = \PolId(\relation{R_3}, \relation{R}_4,\dots,\leq)$.
\item ($\Ac$) Let $\Ac$ be the clone generated by the binary function $\wedge$, where $\wedge$ returns the minimum of its two arguments.
\item ($\Cc$) Let $\Cc = \PolId(\neq)$.
\item ($\Cc^{\leq}$) Let  $\Cc^{\leq} = \PolId(\neq,\leq)$.
\item ($\Lc$) Let $\Lc$ denote the clone consisting of all idempotent linear functions, that is, functions of the form $a_1x_1+a_2x_2+\cdots+a_nx_n,$ where an odd number of the coefficients $a_i$ are equal to $1$ and the remaining coefficients are equal to $0$.
\end{itemize}
\end{definition}

\begin{definition}[$\UP$ set]
Let $S\subseteq \{0,1\}^n$. We equip $\{0,1\}^n$ with the componentwise order, that is
$$(x_1,x_2,\dots, x_n) \leq (y_1,y_2,\dots y_n) \quad \Leftrightarrow \quad x_1\leq y_1 \wedge x_2\leq y_2 \wedge \cdots\wedge x_n\leq y_n$$
We define $\UP(S)=\{x\in\{0,1\}^n : \exists y\in S \text{ such that } x\geq y\}$.
\end{definition}

The following characterizations and properties of the clones will be used repeatedly. Let $f$ be a Boolean function.
\begin{itemize}
\item ($\Bc_n$) 
We have $f\in\Bc_n$ if and only if, for every subset $S\subseteq f^{-1}(1)$ with $|S|\leq n$, there exists a coordinate in which all elements of $S$ are equal to $1$.

\item ($\Bc_\infty$) We have $f\in\Bc_\infty$ if and only if there exists a coordinate such that every element of $f^{-1}(1)$ has value $1$ in that coordinate.
\item ($\Bc_n^{\leq}$) if $f\in\Bc_n$, then the function $g$ defined by $g(x) = 1$ iff $x\in \UP(f^{-1}(1))$ belongs to $\Bc_n^{\leq}$.

\item ($B_\infty^{\leq}$) if $f\in\Bc_\infty$, then the function $g$ defined by $g(x) = 1$ iff $x\in \UP(f^{-1}(1))$ belongs to $\Bc_\infty^{\leq}$.

\item ($\Cc$) We have $f\in \Cc$ if and only if $f(\overline{x}) = \overline{f(x)}$ for every $x$.
\end{itemize}
    Furthermore, a {\em minority} is a function of odd arity that returns the value occurring fewer times among its inputs, provided that it occurs at least once. Similarly, a {\em majority} is a function of odd arity that returns the value occurring more frequently among its inputs.
\begin{itemize}
\item ($\Cc$) Both the minority and the majority belong to $\Cc$.
\item ($\Cc^{\leq}$) The clone $\Cc^{\leq}$ consists precisely of those functions in $\Cc$ that are monotone.
\item ($\Cc^{\leq}$) In particular, the majority function belongs to $\Cc^{\leq}$.
\end{itemize}

\subsection{Products modulo trivial equivalence}

Figure~\ref{lattice} shows the poset of all products of Boolean clones modulo trivial equivalence, ordered by the relation $\preceq$. This poset can be viewed as the free meet-completion of the poset shown in Figure~\ref{palanda}. Observe that the latter is a complete lattice and can moreover be covered by four chains. It follows that its free meet-completion, depicted in Figure~\ref{lattice}, is also a complete lattice. In particular, allowing infinite meets (products) does not yield any additional elements. 

The red vertices represent all products that are trivially equivalent to clones appearing in Figure~\ref{palanda}. We first consider products of the form $\C{\Bc_i,\Bc_j^{\leq}}$ with $i<j$. These products form a triangular subposet. Adding the clone $\Ac$ contributes only a single additional element below this triangle, since $\Ac$ is comparable with every other of these elements. 

Next, we incorporate the remaining clones from Figure~\ref{palanda}, namely $\Cc$, $\Cc^{\leq}$, and $\Lc$. These clones give rise to four distinct products  $\C{\Cc}, \C{\Cc^{\leq}}, \C{\Lc}, \C{\Lc,\Cc^{\leq}}.$ Combining each of these products with the triangular subposet described above yields four additional copies of the triangle. This results in the structure depicted in Figure~\ref{lattice}.

\begin{figure}
\centering{
{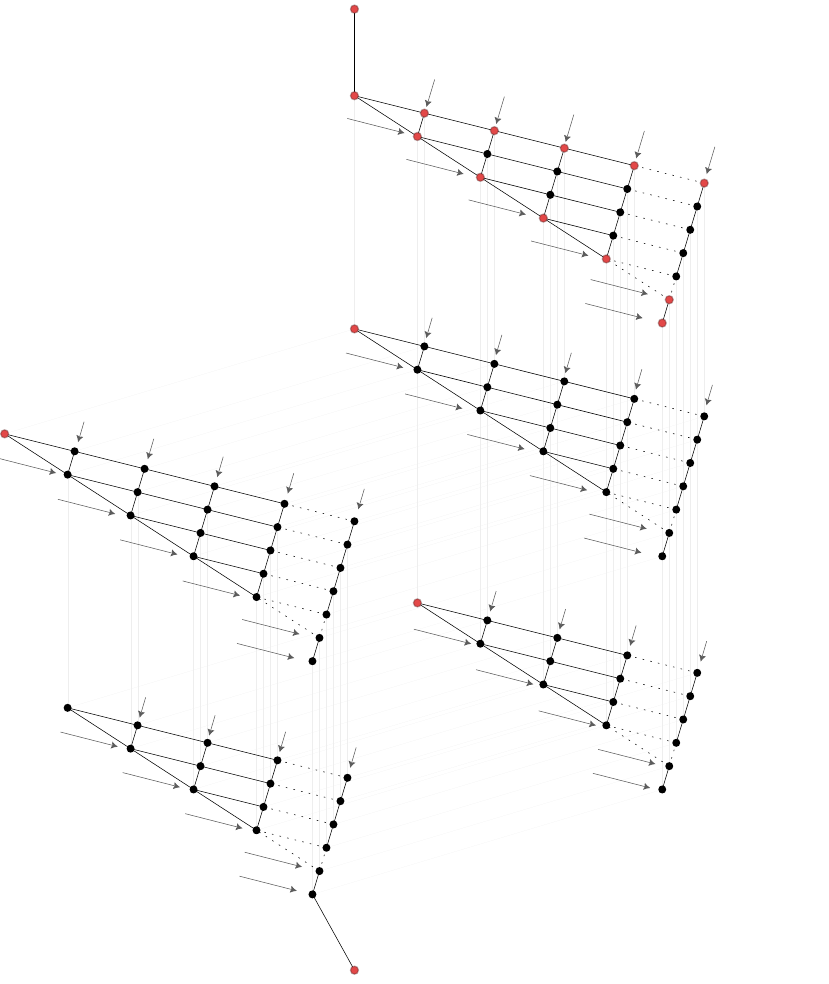}
\caption{Lattice of products of Boolean clones modulo trivial equivalence}
\label{lattice}
}
\end{figure}

\subsection{Orders coincide on specific pairs}

We now work toward proving that $\preceq$ and $\leq$ coincide on all products of Boolean clones. In particular, Figure~\ref{lattice} also represents the minion homomorphism poset restricted to products of Boolean clones.

By Lemma~\ref{ineq-triv}, we already know that $\clone{X} \preceq \clone{Y}$ implies $\clone{X} \leq \clone{Y}$. We now establish the converse for selected pairs $(\clone{X},\clone{Y})$, showing that $\clone{X}\not\preceq\clone{Y}$ implies $\clone{X}\not\leq\clone{Y}$.

All lemmas in this section follow the same proof strategy. For each non-inequality, we exhibit a specific set of height-one identities and then apply Lemma~\ref{identity-proof-strategy} to conclude the desired non-inequality. Recall that, by definition, a product of clones contains a function satisfying a given collection of identities if and only if each of its components contains a function satisfying those identities. 

Each lemma establishes one or more non-inequalities of the form $\clone{X}\not\leq \clone{Y}$ where $\clone{X}$ and $\clone{Y}$ are products of clones. To the right of each lemma, we include a miniature version of Figure~\ref{palanda}: the clones forming the product $\clone{X}$ are highlighted in green, while the clone $\clone{Y}$ is highlighted in red. 

We use the following notation throughout this section. Instead of writing $f(x_1,x_2,\dots,x_n),$ we omit the commas and write $f(x_1x_2\dots x_n).$ Furthermore, when it is convenient to visually separate parts of the input, we insert small spaces. For example, $f(100\ 100)=f(100100)=f(1,0,0,1,0,0).$ 

We also use a matrix notation for finite sets of tuples, in which the tuples are written as rows. For example, $\{011,100\} = \left\{\begin{matrix}011\cr 100\end{matrix}\right\}$

\begin{lemma} 
\label{B2}
  $\C{\Lc}\not\leq \C{\Bc_2}$
\end{lemma}
\begin{cleanproof}
\begin{minipage}[t]{0.75\textwidth}\proof We use the identities  known as {\em quasi-Malcev identities}.

{\colorxy 
\begin{align}\label{eq:B2}
f(yyx) =f(xyy) = f(xxx)
\end{align}
}

\begin{itemize}
\item $\Lc \models (\ref{eq:B2})$. Let $f$ be the sum of all variables modulo $2$.
\end{itemize}
\end{minipage}
\hfill
\begin{minipage}[t]{0.25\textwidth}
\vspace{0pt}
\centering
\includegraphics[width=\miniaturewidth]{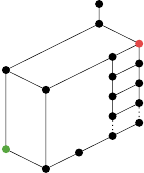}
\end{minipage}
\begin{itemize}
\item $\Bc_2 \not\models (\ref{eq:B2})$. Any function $f$ satisfying (\ref{eq:B2}) must satisfy
\begin{align*}
f(001) &= f(111) = 1\cr
f(100) &= f(111) = 1.
\end{align*}
However, this contradicts the fact that $f$ is compatible with the $\relation{R}_2$ relation. 
\end{itemize}
Lemma~\ref{identity-proof-strategy} therefore implies that $\C{\Lc}\not\leq \C{\Bc_2}$.
\end{cleanproof}

\begin{lemma} 
\label{B2leq}
  $\C{\Lc,\Bc_\infty}\not\leq \C{\Bc_2^{\leq}}$
\end{lemma}
\begin{cleanproof}
\begin{minipage}[t]{0.75\textwidth}\proof We use the identities
{\colorxy
\begin{align}\label{eq:B2leq}
f(y\ xy\ xy) &= f(y\ yx\ yx)\cr
f(y\ xy\ yy) &= f(x\ xx\ xx)
\end{align}
}
\begin{itemize}
\item $\C{\Lc} \models (\ref{eq:B2leq})$. Let $f$ be the sum of all variables modulo $2$.
\item $\C{\Bc_\infty} \models (\ref{eq:B2leq})$. Define $f$ by
$$f^{-1}(1) = \{1\ 11\ 11, 0\ 10\ 00\}$$
\item $\Bc^{\leq}_2 \not\models (\ref{eq:B2leq})$. By the second identity,
$$f(0\ 10\ 00) = f(1\ 11\ 11) = 1$$
Since $f$ is monotone, it follows that
$$f(0\ 10\ 10) = 1.$$

Similarly,
$$f(1\ 01\ 11) = f(0\ 00\ 00) = 0$$
and monotonicity yields
$$f(0\ 01\ 01) = 0$$

These results together contradict the first identity.
\end{itemize}

\end{minipage}
\begin{minipage}[t]{0.25\textwidth}
\vspace{0pt}
\centering
\includegraphics[width=\miniaturewidth]{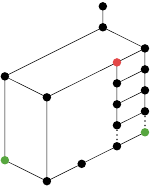}
\end{minipage}
Lemma~\ref{identity-proof-strategy} therefore implies that $\C{\Lc,\Bc_\infty}\not\leq \C{\Bc_2^{\leq}}$.
\end{cleanproof}

\begin{lemma} 
\label{Bi}
  If $i\geq 3$, then $\C{\Lc,\Cc^{\leq},\Bc^{\leq}_{i-1}}\not\leq \C{\Bc_{i}}$
\end{lemma}
\begin{cleanproof}
\begin{minipage}[t]{0.75\textwidth}\proof

We distinguish the cases according to the parity of $i$.

\begin{itemize}
\item Suppose first that $i$ is odd. We use the identities
{\colorxy
\begin{align}\label{eq:Biodd}
f&(y\ {\overbrace{yy\cdots yyx}^{i}}\ {\overbrace{xx\cdots xxy}^{i}})=\cr
=f&(y\ yy\cdots yxy\ xx\cdots xyx)=\cr
&\kern0.2cm\vdots\kern0.5cm\adots\kern1.2cm\adots\cr
=f&(y\ xy\cdots yyy\ yx\cdots xxx)=\cr
=f&(x\ xx\cdots xxx\ xx\cdots xxx)
\end{align}
}
\end{itemize}
\end{minipage}
\begin{minipage}[t]{0.25\textwidth}
\vspace{0pt}
\centering
\includegraphics[width=\miniaturewidth]{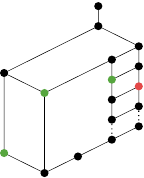}
\end{minipage}
\begin{itemize}
  \item[]
\begin{itemize}
\item $\C{\Lc}\models (\ref{eq:Biodd})$. Let $f$ be the sum of all variables modulo $2$.
\item $\C{\Cc^{\leq}}\models (\ref{eq:Biodd})$. Let $f$ be the majority function on the last $i$ variables.
\item $\C{\Bc_{i-1}^{\leq}}\models (\ref{eq:Biodd})$. Define $f$ by
$$f^{-1}(1) =
\UP\smatrix{0\ {\overbrace{00\cdots 001}^{i}}\ {\overbrace{11\cdots 110}^{i}}\cr
0\ 00\cdots 010\ 11\cdots 101\cr
0\ 00\cdots 100\ 11\cdots 011\cr
\kern-0.1cm\vdots\kern0.5cm\adots\kern1.1cm\adots\kern0.4cm\cr
0\ 10\cdots 000\ 01\cdots 111}$$
\item $\C{\Bc_{i}}\not\models (\ref{eq:Biodd})$. Any function $f$ satisfying (\ref{eq:Biodd}) must satisfy

\begin{align*}
f(1\ {\overbrace{11\cdots 111}^{i}}\ {\overbrace{11\cdots 111}^{i}}) = f(0\ {\overbrace{00\cdots 001}^{i}}\ {\overbrace{11\cdots 110}^{i}}) &= 1 =\cr
 f(0\ 00\cdots 010\ 11\cdots 101) &= 1 =\cr
 \vdots\kern0.5cm\adots\kern1.1cm\adots\kern0.7cm&\cr
 f(0\ 10\cdots 000\ 01\cdots 111) &= 1 =\cr            
\end{align*}
However, this contradicts the fact that $f$ is compatible with the $\relation{R}_i$ relation. 
\end{itemize}

\item Suppose now that $i$ is even. We use the identities
{\colorxy
\begin{align}\label{eq:Bieven}
f&(y\ {\overbrace{xx\cdots xxy}^{i}})=\cr
=f&(y\ xx\cdots xyx)=\cr
&\kern0.2cm\vdots\kern0.55cm\adots\cr
=f&(y\ yx\cdots xxx)=f(x\ xx\cdots xxx)\cr
\end{align}
}

\begin{itemize}
\item $\C{\Lc}\models (\ref{eq:Bieven})$.Let $f$ be the sum of all variables modulo $2$.
\item $\C{\Cc^{\leq}}\models (\ref{eq:Bieven})$.Let $f$ be the majority function on all variables.
\item $\C{\Bc_{i-1}^{\leq}}\models (\ref{eq:Bieven})$. Define $f$ by
$$f^{-1}(1) =
\UP\smatrix{0\ {\overbrace{11\cdots 110}^{i}}\cr
0\ 11\cdots 101\cr
0\ 11\cdots 011\cr
\kern-0.1cm\vdots\kern0.5cm\adots\kern0.4cm\cr
0\ 01\cdots 111}$$

\item $\C{\Bc_{i}}\not\models (\ref{eq:Bieven})$. Any function $f$ satisfying (\ref{eq:Bieven}) must satisfy
\begin{align*}
f(1\ {\overbrace{11\cdots 111}^{i}}) = f(0\ {\overbrace{11\cdots 110}^{i}}) &= 1 =\cr
 f(0\ 11\cdots 101) &= 1 =\cr
 \vdots\kern0.55cm\adots\kern0.7cm&\cr
 f(0\ 01\cdots 111) &= 1 =\cr            
\end{align*}
However, this contradicts the fact that $f$ is compatible with the $\relation{R}_i$ relation. 
\end{itemize}
\end{itemize}
Lemma~\ref{identity-proof-strategy} therefore implies that $\C{\Lc,\Cc^{\leq},\Bc^{\leq}_{i-1}}\not\leq \C{\Bc_{i}}$.
\end{cleanproof}

\begin{lemma} 
\label{Bileq}
  For every $i\geq 3$ it holds that $\C{\Lc,\Cc^{\leq}, \Bc^{\leq}_{i-1}, \Bc_{\infty}}\not\leq \C{\Bc^{\leq}_{i}}$
\end{lemma}
\begin{cleanproof}
\begin{minipage}[t]{0.75\textwidth}\proof

We distinguish the cases according to the parity of $i$.
    
\begin{itemize}
\item Suppose first that $i$ is odd. We use the identities
{\colorxy
\begin{align*}
f&({\overbrace{yy\cdots yyx}^{i}}\ {\overbrace{xx\cdots xxy}^{i}}\ {\overbrace{xx\cdots xxy}^{i}}) =\cr
=f&(yy\cdots yxy\ xx\cdots xyx\ xx\cdots xyx) =\cr
&\kern0.6cm\adots\kern1.1cm\adots\kern1.1cm\adots\cr
=f&(xy\cdots yyy\ yx\cdots xxx\ yx\cdots xxx) 
\end{align*}
}
\end{itemize}
\end{minipage}
\begin{minipage}[t]{0.25\textwidth}
\vspace{0pt}
\centering
\includegraphics[width=\miniaturewidth]{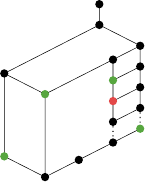}
\end{minipage}

\begin{itemize}
\item[]
{\colorxy
\begin{align}\label{eq:Bileqodd}
f&({\overbounded{yy\cdots yyx}^{i}}\ {\overbounded{xx\cdots xxy}^{i}}\ {\overbounded{yy\cdots yyy}^{i}}) =\cr
=f&({xx\cdots xxx}\ {xx\cdots xxx}\ {xx\cdots xxx})
\end{align}
}
\begin{itemize}
\item $\C{\Lc}\models (\ref{eq:Bileqodd})$. Let $f$ be the sum of all variables modulo $2$.
\item $\C{\Cc^{\leq}}\models (\ref{eq:Bileqodd})$. Let $f$ be the majority function on the middle $i$ variables.
\item $\C{\Bc^{\leq}_{i-1}}\models (\ref{eq:Bileqodd})$. Define $f$ by
$$f^{-1}(1) =
\UP\smatrix{{\overbrace{00\cdots 001}^{i}}\ {\overbrace{11\cdots 110}^{i}}\ {\overbrace{00\cdots 000}^{i}}\cr
00\cdots 010\ 11\cdots 101\ 11\cdots 101\cr
00\cdots 100\ 11\cdots 011\ 11\cdots 011\cr
\kern0cm\adots\kern1cm\adots\kern1cm\adots\cr
10\cdots 000\ 01\cdots 111\ 01\cdots 111}$$
\item $\C{B_\infty}\models (\ref{eq:Bileqodd})$. Define $f$ by
$$f^{-1}(1) = \smatrix{{\overbrace{11\cdots 111}^{i}}\ {\overbrace{11\cdots 111}^{i}}\ {\overbrace{11\cdots 111}^{i}}\cr {00\cdots 001}\ {11\cdots 110}\ {00\cdots 000}}$$
\item $\C{B^{\leq}_{i}}\not\models (\ref{eq:Bileqodd})$. Any function $f$ satisfying (\ref{eq:Bileqodd}) must satisfy
\begin{align*}
  1 &= f({\overbrace{11\cdots 111}^{i}}\ {\overbrace{11\cdots 111}^{i}}\ {\overbrace{11\cdots 111}^{i}}) \cr
  &= f({00\cdots 001}\ {11\cdots 110}\ {00\cdots 000})
\end{align*}
Since $f$ is monotone, it follows that
\begin{align*}  
  f({\overbrace{00\cdots 001}^{i}}\ {\overbrace{11\cdots 110}^{i}}\ {\overbrace{00\cdots 000}^{i}}) &= 1\cr
  \Downarrow\hskip2.3cm&\cr
  f({00\cdots 001}\ {11\cdots 110}\ {11\cdots 110}) &= 1
\end{align*}
Applying (\ref{eq:Bileqodd}) yields
\begin{align*}
f&({\overbrace{00\cdots 001}^{i}}\ {\overbrace{11\cdots 110}^{i}}\ {\overbrace{11\cdots 110}^{i}})=1\cr
f&(00\cdots 010\ 11\cdots 101\ 11\cdots 101)=1\cr
&\kern0.5cm\adots\kern1.1cm\adots\kern1.1cm\adots\kern0.75cm\vdots\cr
f&(10\cdots 000\ 01\cdots 111\ 01\cdots 111)=1
\end{align*}
However, this contradicts the fact that $f$ is compatible with the $\relation{R}_i$ relation. 
\end{itemize}

\item Suppose now that $i$ is even. We use the identities
{\colorxy
\begin{align*}
f&(y\ {\overbrace{xx\cdots xxy}^{i}}\ {\overbrace{xx\cdots xxy}^{i}})=\cr
=f&(y\ xx\cdots xyx\ xx\cdots xyx)=\cr
&\kern0.2cm\vdots\kern0.55cm\adots\kern1.2cm\adots\cr
=f&(y\ yx\cdots xxx\ yx\cdots xxx)
\end{align*}
}

{\colorxy
\begin{align}\label{eq:Bileqeven}
f(y\ {\overbrace{xx\cdots xxy}^{i}}\ {\overbrace{yy\cdots yyy}^{i}}) = f(x\ {\overbrace{xx\cdots xxx}^{i}}\ {\overbrace{xx\cdots xxx}^{i}})
\end{align}
}
\begin{itemize}
\item $\C{\Lc}\models (\ref{eq:Bileqeven})$. Let $f$ be the sum of all variables modulo $2$.
\item $\C{\Cc^{\leq}}\models (\ref{eq:Bileqeven})$. Let $f$ be the majority function on the first $i+1$ variables.
\item $\C{\Bc^{\leq}_{i-1}}\models (\ref{eq:Bileqeven})$. Define $f$ by
$$f^{-1}(1) =
\UP\smatrix{0\ {\overbrace{11\cdots 110}^{i}}\ {\overbrace{00\cdots 000}^{i}}\cr
0\ 11\cdots 101\ 11\cdots 101\cr
0\ 11\cdots 011\ 11\cdots 011\cr
\kern-0.1cm\vdots\kern0.5cm\adots\kern1.1cm\adots\kern0.4cm\cr
0\ 01\cdots 111\ 01\cdots 111}$$
\item $\C{\Bc_\infty}\models (\ref{eq:Bileqeven})$. Define $f$ by
$$f^{-1}(1) = \smatrix{1\ {\overbrace{11\cdots 111}^{i}}\ {\overbrace{11\cdots 111}^{i}}\cr {0}\ {11\cdots 110}\ {00\cdots 000}}$$
\item $\C{\Bc^{\leq}_{i}}\not\models (\ref{eq:Bileqeven})$. Any function $f$ satisfying (\ref{eq:Bileqeven}) must satisfy
$$1 = f(1\ {\overbrace{11\cdots 111}^{i}}\ {\overbrace{11\cdots 111}^{i}}) = f(0\ {\overbrace{11\cdots 110}^{i}}\ {\overbrace{00\cdots 000}^{i}})$$
Since $f$ is monotone, it follows that
\begin{align*}
  f(0\ {\overbrace{11\cdots 110}^{i}}\ {\overbrace{00\cdots 000}^{i}})&=1 \cr
  \Downarrow\hskip1.5cm&\cr
  f(0\ {11\cdots 110}\ {11\cdots 110}) &= 1
\end{align*}
Applying (\ref{eq:Bileqeven}) yields
\begin{align*}
f&(0\ {\overbrace{11\cdots 110}^{i}}\ {\overbrace{11\cdots 110}^{i}})=1\cr
f&(0\ 11\cdots 101\ 11\cdots 101)=1\cr
&\vdots\kern0.7cm\adots\kern1.2cm\adots\kern0.5cm\cr
f&(0\ 01\cdots 111\ 01\cdots 111)=1
\end{align*}
However, this again contradicts the fact that $f$ is compatible with the $\relation{R}_i$ relation. 
\end{itemize}
\end{itemize}
Lemma~\ref{identity-proof-strategy} therefore implies that $\C{\Lc,\Cc^{\leq}, \Bc^{\leq}_{i-1}, \Bc_{\infty}}\not\leq \C{\Bc^{\leq}_{i}}$.
\end{cleanproof}

\begin{lemma} 
\label{A}
  $\C{\Lc,\Cc^{\leq},\Bc_{\infty}^{\leq}}\not\leq \C{\Ac}$
\end{lemma}
\begin{cleanproof}
\begin{minipage}[t]{0.75\textwidth}\proof We use the identities

{\colorxy
\begin{align}\label{eq:A}
p_1(xxy) &= p_1(xxx)\cr
p_2(xxy) &= p_1(xyy)\cr
q(xyy) &= p_2(xyy)\cr
q(xxy) &= q(yyy) \cr
p_i(xyx) &= p_i(xxx) \quad {i\in\{1,2\}}\cr
\end{align}
}

\end{minipage}
\begin{minipage}[t]{0.25\textwidth}
\vspace{0pt}
\centering
\includegraphics[width=\miniaturewidth]{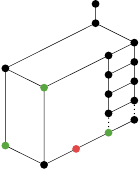}
\end{minipage}
\begin{itemize}
\item $\C{\Lc}\models (\ref{eq:A})$. Let $p_1(xyz)=p_2(xyz)=x$ and let $q(xyz)=x+y+z$. Then
{\colorxy
\begin{align*}
p_1(xxy) &= p_1(xxx) = x\cr
p_2(xxy) &= p_1(xyy) = x\cr
q(xyy) &= p_2(xyy) = x\cr
q(xxy) &= q(yyy) = y\cr
p_i(xyx) &= p_i(xxx) = x.
\end{align*}
}
\item $\C{\Cc^{\leq}}\models (\ref{eq:A})$. Let $p_1(xyz)$ be the majority function and let $p_2(xyz)=q(xyz)=z$. Then
{\colorxy 
\begin{align*}
p_1(xxy) &= p_1(xxx) = x\cr
p_2(xxy) &= p_1(xyy) = y\cr
q(xyy) &= p_2(xyy) = y\cr
q(xxy) &= q(yyy) = y\cr
p_i(xyx) &= p_i(xxx) = x.
\end{align*}
}
\item $\C{\Bc_{\infty}^{\leq}} \models (\ref{eq:A})$. Let $p_1(xyz) = x\wedge (y\vee z)$, $p_2(xyz) = z\wedge (x\vee y)$ and let $q(xyz) = z$. Then 

{\colorxy 
\begin{align*}
p_1(xxy) &= p_1(xxx) = x\cr
p_2(xxy) &= p_1(xyy) = x\wedge y\cr
q(xyy) &= p_2(xyy) = y\cr
q(xxy) &= q(yyy) = y\cr 
p_i(xyx) &= p_i(xxx) = x
\end{align*}
}
\item  $\C{\Ac}\not\models (\ref{eq:A})$. From the identity $\colorxy q(xxy)=q(yyy)=y$ we obtain $q(xyz) = z$. Now consider the identities
{\colorxy
\begin{align*}
y = q(xyy) &= p_2(xyy)\cr
p_2(xyx) &= p_2(xxx) = x,
\end{align*}
}
From ${\colorxy p_2(xyy) = q(xyy)=y}$, it follows that, among the meet terms, $p_2(xyz)$ is either $y$, $z$ or $y\wedge z$. 
Substituting these candidates into ${\colorxy p_2(xyx)=p_2(xxx)=x}$
shows that the only possibility is $p_2(xyz) = z$. 

Next consider
{\colorxy
\begin{align*}
y = p_2(xxy) &= p_1(xyy)\cr
p_1(xyx) &= p_1(xxx) = x,
\end{align*}
}
Similarly from ${\colorxy p_1(xyy) = p_2(xxy)=y}$, it follows that, among the meet terms,
 $p_1(xyz)$ is either $y$, $z$ or $y\wedge z$. Substituting these candidates into ${\colorxy p_1(xyx) = p_1(xxx) = x}$ we get $p_1(xyz)=z$. However, this contradicts the identity $\colorxy p_1(xxy)=p_1(xxx)=x$.
\end{itemize}
Lemma~\ref{identity-proof-strategy} therefore implies that $\C{\Lc,\Cc^{\leq},\Bc_{\infty}^{\leq}}\not\leq \C{\Ac}$.
\end{cleanproof}

\begin{lemma}
\label{C}
  $\C{\Ac}\not\leq\C{\Cc}$
\end{lemma} 
\begin{cleanproof}
\begin{minipage}[t]{0.75\textwidth}\proof We use the identity

{\colorxy
  \begin{align}\label{eq:C}
    f(xy)=f(yx)
  \end{align}
}
  
\begin{itemize}
\item $\C{\Ac}\models (\ref{eq:C})$. Let $f(xy) = x\wedge y$. 
\item $\C{\Cc}\not\models (\ref{eq:C})$. Suppose that $f$ satisfies (\ref{eq:C}). Then 
$f(10) = f(01)$ \\ 
On the other hand preservation of $\neq$ implies that $f(10) \neq f(01)$, \\ 
which is a contradiction.
\end{itemize}
\end{minipage}
\begin{minipage}[t]{0.25\textwidth}
\vspace{0pt}
\centering
\includegraphics[width=\miniaturewidth]{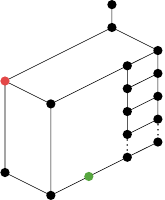}
\end{minipage}
\vskip2pt
\parindent=0pt Lemma~\ref{identity-proof-strategy} therefore implies that $\C{\Ac}\not\leq\C{\Cc}$.
\end{cleanproof}

\begin{lemma}
\label{Cleq}
  $\C{\Lc,\Ac}\not\leq \C{\Cc^{\leq}}$
\end{lemma}
\begin{cleanproof}
\begin{minipage}[t]{0.75\textwidth}\proof We use the identities

 {\colorxy
  \begin{align}\label{eq:Cleq}
    f(xyyyy)=f(yxxxy)=f(xyyxx)
  \end{align}
}

\begin{itemize}
\item $\C{\Lc}\models (\ref{eq:Cleq})$. Let $f$ be the sum of all variables modulo $2$.
\item $\C{\Ac}\models (\ref{eq:Cleq})$. Let $f$ be the meet of all variables.
\item $\C{\Cc^{\leq}}\not\models (\ref{eq:Cleq})$. Suppose that such a function $f$ exists. We distinguish two cases.
\begin{itemize}
\item Suppose that $f(10000)=1$. Since $f$ is compatible with $\neq$ we obtain $f(01111)=0$. Further since $f$ is monotone, we obtain $$f(01110)\leq f(01111)=0,$$ and therefore $$f(01110)=0\neq 1=f(10000).$$ This however contradicts \ref{eq:Cleq}.
\end{itemize}
\end{itemize}
\end{minipage}
\begin{minipage}[t]{0.25\textwidth}
\vspace{0pt}
\centering
\includegraphics[width=\miniaturewidth]{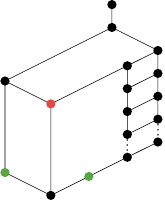}
\end{minipage}
\begin{itemize}
  \item[]
\begin{itemize}
\item Suppose that $f(10000)=0$. By (\ref{eq:Cleq}) we get $f(01110)=0$. Since $f$ preserves $\neq$, it follows that $f(10001)=1$. Using monotonicity, we obtain $$f(10011)\geq f(10001)=1,$$ and hence
$$f(10011)=1\neq 0=f(10000).$$
This however contradicts \ref{eq:Cleq}.
\end{itemize}
\end{itemize}
Lemma~\ref{identity-proof-strategy} therefore implies that $\C{\Lc,\Ac}\not\leq \C{\Cc^{\leq}}$.
\end{cleanproof}

\begin{lemma} 
\label{L}
  $\C{\Cc^{\leq},A}\not\leq \C{\Lc}$
\end{lemma}
\begin{cleanproof}
\begin{minipage}[t]{0.75\textwidth}\proof We use the identities

{\colorxy
\begin{align}\label{eq:L}
1.\quad f(yyxxx) &= f(xyxxx)\cr
2.\quad f(xyyxx) &= f(xxyxx)\cr
3.\quad f(xxyyx) &= f(xxxyx)\cr
4.\quad f(xxxyy) &= f(xxxxy)\cr
5.\quad f(yxxxy) &= f(yxxxx)
\end{align}
}

\end{minipage}
\begin{minipage}[t]{0.25\textwidth}
\vspace{0pt}
\centering
\includegraphics[width=\miniaturewidth]{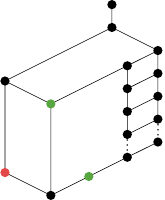}
\end{minipage}
\begin{itemize}
\item $\C{\Cc^{\leq}}\models(\ref{eq:L})$. Let $f$ be the majority function on all five variables.
\item $\C{\Ac}\models(\ref{eq:L})$. Let $f$ be the meet of all five variables.
\item $\C{\Lc}\not\models(\ref{eq:L})$. Suppose, for a contradiction, that $f$ is a linear function satisfying (\ref{eq:L}). Since $f$ is nonconstant, there exists a variable on which $f$ depends. Then the $i$-th identity in (\ref{eq:L}) changes only the $i$-th coordinate while leaving all remaining coordinates unchanged. Evaluating $f$ on the two sides of this identity therefore yields values differing by $1$, contradicting the identity. 
\end{itemize}
Lemma~\ref{identity-proof-strategy} therefore implies that $\C{\Cc^{\leq},A}\not\leq \C{\Lc}$.
\end{cleanproof}

\begin{lemma}
\label{idemp}
  $\Mc\not\leq \Ic,$ where $\Mc$ is the Boolean clone consisting of all functions and $\Ic$ is the clone consisting of all idempotent functions.
\end{lemma}
\begin{cleanproof}
\begin{minipage}[t]{0.75\textwidth}\proof We use the identity
{\colorxy
\begin{align}\label{eq:idem}    
f(x) = f(y)
\end{align}
}
\begin{itemize}
\item $\Mc\models (\ref{eq:idem})$. Any constant function satisfies the identity.
\item $\Ic\not\models (\ref{eq:idem})$. Suppose that $f$ is an idempotent unary function. Then $f(x) = x\neq y = f(y)$. \\ Therefore no such $f$ exists.
\end{itemize}
\end{minipage}
\begin{minipage}[t]{0.25\textwidth}
\vspace{0pt}
\centering
\includegraphics[width=\miniaturewidth]{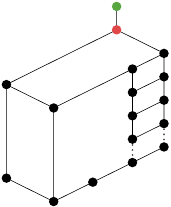}
\end{minipage}
Lemma~\ref{identity-proof-strategy} therefore implies that $\Mc\not\leq \Ic$.
\end{cleanproof}

\begin{lemma} 
\label{projp}
  $\C{\Lc,\Cc^\leq,\Ac}\not\leq \Pc$, where $\Pc$ is the clone consisting of all projections.
\end{lemma}
\begin{cleanproof}
\begin{minipage}[t]{0.75\textwidth}\proof We use the identities
{\colorxy
\begin{align}\label{eq:proj} 
f(yyyxxxx) &= f(xxyxxxx)\cr
f(xxyyyxx) &= f(xxxxyxx)\cr
f(xxxxyyy) &= f(xxxxxxy)\cr
f(xxxxyyy) &= f(xxxxyxx)\cr
\end{align}
}
\end{minipage}
\begin{minipage}[t]{0.25\textwidth}
\vspace{0pt}
\centering
\includegraphics[width=\miniaturewidth]{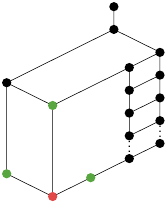}
\end{minipage}

\begin{itemize}
\item $\Lc\models (\ref{eq:proj})$. Let $f$ be the sum of all variables modulo $2$.
\item $\Cc^\leq\models (\ref{eq:proj})$. Let $f$ be the majority function on all variables.
\item $\Ac\models (\ref{eq:proj})$. Let $f$ be the meet of all variables.
\item $\Pc\not\models (\ref{eq:proj})$. We show that no projection can satisfy all identities in (\ref{eq:proj}).
\begin{itemize}
\item  Projection onto the first or second coordinate contradicts the first identity.
\item  Projection onto the third or fourth coordinate contradicts the second identity
\item  Projection onto the fifth or sixth coordinate contradicts the third identity.
\item  Projection onto the seventh coordinate contradicts the fourth identity.
\end{itemize}
\end{itemize}
Lemma~\ref{identity-proof-strategy} therefore implies that $\C{\Lc,\Cc^\leq,\Ac}\not\leq \Pc$.
\end{cleanproof}

\subsection{Orders coincides on all pairs}

We now complete the proof that the orders $\preceq$ and $\leq$ coincide.  Specifically, we show that
$$\clone{X}\not\preceq\clone{Y}\quad\Rightarrow\quad\clone{X}\not\leq\clone{Y}$$
for all pairs $(\clone{X},\clone{Y})$, and not only for those considered above. To this end, we require a few additional definitions and lemmas from lattice theory~\cite{trotter}.

\begin{definition}[critical pair] Let $L$ be a lattice with order relation $\preceq$, and let $x,y\in L$. We say that the ordered pair $(x,y)$ is {\em critical} if it satisfies the following conditions.\footnote{The standard definition additionally requires $x$ and $y$ to be incomparable. We omit this requirement here.}
\begin{itemize}
\item $x$ is minimal among the elements $\{e: e\not\preceq y\}$.
\item $y$ is maximal among the elements $\{e: e\not\succeq x\}$.
\end{itemize}
\end{definition}

\begin{definition}[meet irreducible] Let $L$ be a lattice. An element $x\in L$ is called {\em meet irreducible} if, whenever $x=\bigwedge S$, it follows that $x\in S$.
\end{definition}

\begin{definition}[join irreducible] Let $L$ be a lattice. An element $x\in L$ is called {\em join irreducible} if, whenever $x=\bigvee S$, it follows that $x\in S$.
\end{definition}

\begin{lemma}
\label{meet-irred}
  Let $x\in L$, and let $y$ be a maximal element of the set $\{s:x\not\leq s\}$. Then $y$ is meet irreducible.
\end{lemma}

\begin{lemma} 
\label{critical}
  Suppose that $(x,y)$ is a critical pair in a lattice $L$. Then $x$ is join irreducible and $y$ is meet irreducible.
\end{lemma}

Throughout this section, let $L$ denote the lattice of all products of Boolean clones modulo trivial equivalence, ordered by $\preceq$ and shown in Figure~\ref{lattice}.

\begin{lemma}
\label{allineq}
  Let $\clone{Y}\in L$ be meet irreducible. Then the following statements hold.
\begin{itemize}
\item There exists a unique element $\clone{X}$ such that $(\clone{X},\clone{Y})$ is a critical pair. Fix such an element $\clone{X}$.
\item For every element $\clone{Z}\in L$, either $\clone{Z}\preceq \clone{Y}$ or $\clone{X}\preceq Z$.
\item We have $\clone{X}\not\leq \clone{Y}$.
\end{itemize}
\end{lemma}
\begin{proof} 
No meet irreducible element of $L$ can be represented as a nontrivial product of distinct clones, since such an element would be the meet of its components and therefore would not be meet irreducible.
 Consequently, every meet irreducible element $\clone{Y}\in L$ appears among the clones shown in Figure~\ref{palanda}.
 The only exceptions are
$\Bc_\infty = \bigwedge \Bc_i$, $\Bc_\infty^\leq = \bigwedge \Bc_i^\leq$ which are not meet irreducible. The meet irreducible elements of $L$ are shown in Figure~\ref{palandairred}.

\begin{figure}
\centering{
{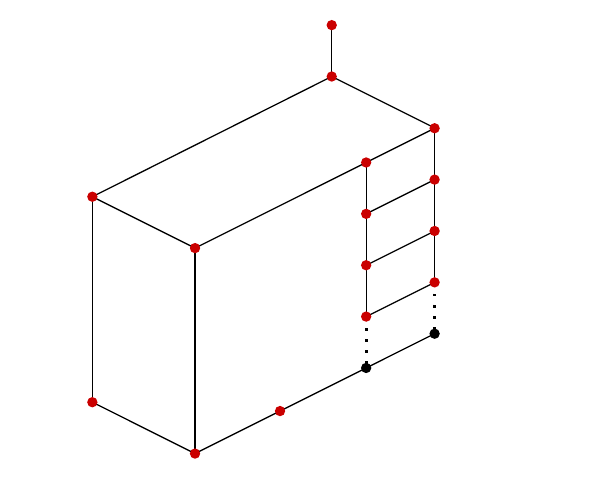}
\caption{Meet irreducible Boolean clones}
\label{palandairred}
}
\end{figure}

For a fixed meet irreducible element $\clone{Y}$, define $\clone{X} = \bigwedge \{\clone{W}:\clone{W}\in L,\clone{W}\not\preceq \clone{Y}\}$. 
Since $L$ is closed under arbitrary meets, the element $\clone{X}$ belongs to $L$. Moreover, every element of $L$ is a product of elements from Figure~\ref{palandairred}. Hence $\clone{X}$ can be expressed as the product of all minimal elements that are not below $\clone{Y}$.

By this construction, every element $\clone{Z}\in L$ satisfies either $\clone{Z}\leq \clone{Y}$ or $\clone{X}\leq \clone{Z}$.
Therefore, $(\clone{X},\clone{Y})$ is a critical pair for $\clone{Y}$. Since $\clone{X}$ is uniquely determined by the above construction, it follows that $(\clone{X},\clone{Y})$ is the unique critical pair whose second component is $\clone{Y}$.

The following table lists, for each meet irreducible element $\clone{Y}$, the unique element $\clone{X}$ such that $(\clone{X},\clone{Y})$ is a critical pair, together with the lemma establishing that $\clone{X}\not\leq\clone{Y}$.

\begin{center}
\begin{tabular}{ |l|l|l| } 
 \hline
 $Y$                       & $X$                                       & $X\not\leq Y$ \\
 \hline
 $\Bc_2$                   & $\Lc$                                     & \ref{B2} \\
 $\Bc_2^\leq$               & $\C{\Lc,\Bc_\infty}$                       & \ref{B2leq} \\
 $\Bc_i$ for $i\geq 3$     & $\C{\Lc,\Cc^\leq,\Bc_{i-1}^\leq}$           & \ref{Bi} \\
 $\Bc_i^\leq$ for $i\geq 3$ & $\C{\Lc,\Cc^\leq,\Bc_\infty,\Bc_{i-1}^\leq}$ & \ref{Bileq} \\
 $\Ac$                     & $\C{\Lc,\Cc^\leq,\Bc_\infty^\leq}$          & \ref{A} \\
 $\Cc$                     & $\Ac$                                     & \ref{C} \\
 $\Cc^\leq$                 & $\C{\Lc,\Ac}$                             & \ref{Cleq} \\
 $\Lc$                     & $\C{\Cc^\leq,\Ac}$                         & \ref{L} \\
 idempotent                & all functions                             & \ref{idemp} \\
 projections               & $\C{\Lc,\Cc^\leq,\Ac}$                     & \ref{projp} \\
 \hline
\end{tabular}
\end{center}
\end{proof}

\begin{lemma}
\label{final}
  For every $\clone{X},\clone{Y}\in L$, if $\clone{X}\not\preceq \clone{Y}$, then $\clone{X}\not\leq \clone{Y}$.
\end{lemma}
\begin{proof}
  Let $\clone{Y}_1$ be the maximum element of the set $\{\clone{S}: \clone{S}\succeq \clone{Y},\ \clone{X}\not\preceq \clone{S}\}$. Such an element exists since $L$ contains no infinite ascending chains. Note that $\clone{Y}_1\succeq \clone{Y}$. 
  
  Then observe that $\clone{Y}_1$ is also the maximum element of the set $\{\clone{S}: \clone{X}\not\preceq \clone{S}\}$, and thus Lemma \ref{meet-irred} implies that $\clone{Y}_1$ is meet irreducible. 

Applying Lemma~\ref{allineq}, let $\clone{X}_1$ be the unique element such that $(\clone{X}_1,\clone{Y}_1)$ is a critical pair. Since $\clone{X}\not\preceq \clone{Y}_1$, Lemma~\ref{allineq} further implies that $\clone{X}_1\preceq \clone{X}$.

Suppose, for the sake of contradiction, that $\clone{X}\leq \clone{Y}$. Since $\clone{X}_1\preceq \clone{X}$ and $\preceq$ implies $\leq$, we obtain $\clone{X}_1\leq \clone{X}$. Together with $\clone{X}\leq\clone{Y}$ and $\clone{Y}\preceq\clone{Y}_1$, which again implies $\clone{Y}\leq\clone{Y}_1$, we get $\clone{X}_1\leq \clone{X}\leq \clone{Y}\leq \clone{Y}_1$. By transitivity, $\clone{X}_1\leq \clone{Y}_1$.

However, Lemma~\ref{allineq} states that $\clone{X}_1\not\leq \clone{Y}_1$, yielding a contradiction. Therefore, $\clone{X}\not\leq \clone{Y}$.
\end{proof}

\begin{theorem}  
The minion homomorphism poset on the class of products of Boolean clones is the poset $L$ shown in Figure~\ref{lattice}.
\end{theorem}
\begin{proof}
The orders $\leq$ and $\preceq$ coincide. The lemmas \ref{ineq-triv} and \ref{final} establish the two needed implications. Therefore, the lattice of $\preceq$ shown in Figure \ref{lattice} is precisely the lattice corresponding to $\leq$.
\end{proof}

\section{Cores}

In this section, we determine cores for all clones shown in Figure~\ref{lattice}. To this end, we first introduce some additional terminology.  

\begin{definition}[function minion]
  Let $\set{A}$ and $\set{B}$ be sets. A set of functions $\minion{M}$ from $\set{A}$ to $\set{B}$, of arbitrary arities,  is called a {\em minion} if for every $n$-ary $f\in \minion{M}$, every $m\geq 1$, and every choice of indices $i_1,i_2,\dots,i_n\in [m]$, we have
  $$f\circ(\pi_{i_1}^m,\dots,\pi_{i_n}^m)\in \minion{M}$$
\end{definition}

When discussing minions, we will always mean function minions. As in the case of clones, we will also consider multisorted minions. We therefore restate the definition in the multisorted setting. 

\begin{definition}[multisorted function minion] 
Let $\set{A}=(A_1,\dots,A_k)$ and $\set{B}=(B_1,\dots,B_k)$ be $k$-sorted sets. A {\em multisorted minion} $\minion{M}$ is a collection of $k$-tuples of functions of arbitrary arities such that, for every $n\geq 1$, each $n$-ary element of $\minion{M}$ has the form 
$$(f_1,\dots,f_k),$$ 
where $f_i\colon A_i^n \to B_i$ for all $i\in [k]$. Moreover, $\minion{M}$ is closed under taking minors: for every $n$-ary $f=(f_1,\dots,f_k)\in \minion{M}$ , every $m\geq 1$, and every choice of indices $i_1,\dots,i_n\in [m]$, we have $$f\circ (\pi_{i_1}^m,\dots,\pi_{i_n}^m)\in \minion{M}.$$
\end{definition}

Note that the definition of a minion homomorphism applies to arbitrary minions. Consequently, within a clone equivalence class modulo minion homomorphisms, one may consider minions that are not necessarily clones themselves. Furthermore, we are interested in, in a sense, the smallest minions within these equivalence classes, known as {\em minion cores}. 

\begin{definition}[minion core]
A minion $\minion{M}$ is called a {\em minion core} if every minion endomorphism of $\minion{M}$ is an automorphism.
\end{definition}

\begin{definition}[core of a clone]
  Let $\clone{A}$ be a clone. We say that a minion $\minion{M}$ is a \emph{core} of $\clone{A}$ if $\clone{A}\sim\minion{M}$ and $\minion{M}$ is a minion core.
\end{definition}

We will see that, in the Boolean setting, a minion homomorphism is uniquely determined by its action on binary functions. This yields a finite characterization of minion homomorphisms \cite{kapytka}.

\begin{lemma}
  \label{bin-equal}
Let $\minion{M}$ and $\minion{N}$ be Boolean minions. Let $\psi,\phi\colon \minion{M}\to\minion{N}$ be minion homomorphisms such that $\psi(f)=\phi(f)$ for every binary function $f\in\minion{M}$. Then $\psi=\phi$.
\end{lemma}
\begin{proof}
  Let $f$ be $n$-ary function in $\minion{M}$. We show that $\psi(f)=\phi(f)$. 
  
  Fix a tuple $a=(a_1,a_2,\dots,a_n)\in\{0,1\}^n$. 
  Let $\pi=(p_1,\dots,p_n)$, where 
  $$p_i= \begin{cases} \pi_1^2,& \text{if } a_i=0,\\ \pi_2^2,& \text{if } a_i=1. \end{cases} $$
  Then $f\circ\pi$ is a binary function in $\minion{M}$, and therefore $\psi(f\circ\pi)=\phi(f\circ\pi)$. 
  Since $\psi$ and $\phi$ are minion homomorphisms, they preserve minors. Hence
  $$\psi(f)\circ\pi = \psi(f\circ\pi) = \phi(f\circ\pi) = \phi(f)\circ\pi.$$ Evaluating both sides at $(0,1)$ yields
  $$\psi(f)(a) = \psi(f\circ\pi)(0,1) = \phi(f\circ\pi)(0,1) = \phi(f)(a)$$
Since this holds for every $a\in\{0,1\}^n$, we conclude that $\psi(f)=\phi(f)$. 
\end{proof}

Similarly, the same argument extends to the multisorted setting.

\begin{lemma}
Let $\minion{M}$ and $\minion{N}$ be $k$-sorted Boolean minions. Let $\psi,\phi\colon\minion{M}\to\minion{N}$ be minion homomorphisms such that
 $\psi(f)=\phi(f)$ for every $k$-tuple of binary function $f\in\minion{M}$. Then $\psi=\phi$.
\end{lemma}

Finally, we show a useful criterion for proving that a minion is a minion core.

\begin{lemma}
  Let $\minion{M}$ be a Boolean minion. If every binary function of $\minion{M}$ lies in the image of every endomorphism of $\minion{M}$, then $\minion{M}$ is a minion core.
\end{lemma}
\begin{proof}
  Let $\psi$ be an endomorphism of $\minion{M}$. We show that $\psi$ is bijective.

  By assumption, $\psi$ is surjective on binary functions. Since there are only finitely many binary functions, $\psi$ induces a permutation of them.

  Note that on binary functions $\psi$ is a permutation, hence there exists $n$ such that $\psi^n$ is the identity on binary functions. But then from Lemma \ref{bin-equal} $\psi^n$ is an identity everywhere, but that means that $\psi$ had to be bijective.  
\end{proof}

Similarly, the previous lemma extends to the multisorted setting.

\begin{lemma}
  \label{proof-minion}
  Let $\minion{M}$ be a $k$-sorted Boolean minion. If every $k$-tuple of binary functions in $\minion{M}$ lies in the image of every endomorphism of $\minion{M}$, then $\minion{M}$ is a minion core.
\end{lemma}

For certain special $n$-tuples, we use the following notation:
\begin{align*}
  1_i &= (\overbrace{111\cdots 1}^i\ \overbrace{000\cdots 0}^{n-i}) \cr
  0_i &= (000\cdots 0\ 111\cdots 1) \cr
  x_i &= (xxx\cdots x\ yyy\cdots y) \cr
  y_i &= (yyy\cdots y\ xxx\cdots x) 
\end{align*}
Note that the value of $n$ is omitted from the notation and will always be clear from the context.

\begin{definition}[permuting variables] Let $\pi\colon [n]\to [n]$ be a permutation, and let $a=(a_1,\dots,a_n)$ be a tuple of length $n$. We denote by $\pi(a)$
the tuple obtained by permuting the coordinates of $a$ according to $\pi$, namely
$$\pi(a) = (a_{\pi(1)},a_{\pi(2)},\dots,a_{\pi(n)})$$
\end{definition}

\begin{definition}[$k$-pseudosymmetric] 
Let $f$ be an $n$-ary function. We say that $f$ is {\em $k$-pseudosymmetric} if for every permutation $\pi:[n]\to[n]$,
$$f(x_k)=f(\pi(x_k)).$$
Equivalently, the value of $f$ depends only on the fact that its input contains exactly $k$ ones, and not on the positions of those ones.
\end{definition}

\begin{remark}
By symmetry between $x$ and $y$, an $n$-ary function is $k$-pseudosymmetric if and only if it is $(n-k)$-pseudosymmetric.
\end{remark}

\begin{lemma}
    Let $f\in \Lc$ be $n$-ary function that is $k$-pseudosymmetric for some $0<k<n$. Then $f$ is the sum of all variables.
\end{lemma}
\begin{proof}
    Let $S$ be the set of non-dummy variables of $f$. Since $f$ is idempotent, we have $|S|>0$. Suppose, for a contradiction, that $|S|<n$.
    
    Since $0<k<n$ and $0<|S|<n$, we can choose a set $K$ of variables such that $|K|=k$, $K\cap S\neq \emptyset$, and $|K\cup S| < n$. 
    Let $i\in K\cap S$ and $j\not\in K\cup S$. 
    
    Assign the value $1$ to all variables in $K$ and the value $0$ to all remaining variables. Since $f$ is $k$-pseudosymmetric, the value of $f$ remains unchanged if we replace the $i$-th variable by $0$ and the $j$-th variable by $1$.
    
    However, $i$ is non-dummy and $j$ is dummy. Since $f\in\Lc$, changing the $i$-th variable from $1$ to $0$ changes the value of $f$, whereas changing the $j$-th variable does not. This is a contradiction.

    Therefore $|S|=n$, and hence $f$ is the sum of all variables.
\end{proof}

\begin{lemma}
    Let $f\in \Cc^{\leq}$ be $\lceil{\frac{n}{2}}\rceil$-pseudosymmetric. Then $f$ is the majority on all variables.
\end{lemma}
\begin{proof}
    Since $f$ is self-dual and $\lceil{\frac{n}{2}}\rceil$-pseudosymmetric, $n$ must be odd. Moreover, because $\lceil{\frac{n}{2}}\rceil + \lfloor{\frac{n}{2}}\rfloor = n$ it follows that $f$ is also $\lfloor\frac{n}{2}\rfloor$-pseudosymmetric.

Since $f$ is both $\lfloor\frac{n}{2}\rfloor$-pseudosymmetric and $\lceil\frac{n}{2}\rceil$-pseudosymmetric, there exists a constant $c$ such that $f$ is equal to $c$ on every input containing exactly $\lfloor\frac{n}{2}\rfloor$ ones. By self-duality, $f$ is equal to $\bar c$ on every input containing exactly $\lceil\frac{n}{2}\rceil$ ones.

Monotonicity implies that $c=0$. Consequently, every input containing fewer than $\lfloor\frac{n}{2}\rfloor$ ones is also mapped to $0$. Similarly, we have $\bar c=1$, and every input containing more than $\lceil\frac{n}{2}\rceil$ ones is mapped to $1$. 

Therefore, $f$ evaluates to $1$ precisely on those inputs containing more than half of their coordinates equal to $1$. Hence $f$ is the majority function on all variables.
\end{proof}

\begin{definition}[$\beta(n,k)$]
Let $\beta(n,k) = \lfloor\frac{k-1}{k}\cdot n\rfloor$.
\end{definition}

\begin{lemma}
Let $f$ be an $n$-ary function compatible with $\relation{R}_k$ and let $b \leq \beta(n,k)$. If $f$ is $b$-pseudosymmetric, then $f(1_b) = 0$.
\end{lemma}
\begin{proof}
Suppose, for a contradiction, that $f(1_b)=1$. Since $f$ is $b$-pseudosymmetric, it follows that $f$ evaluates to $1$ on every tuple containing exactly $b$ ones. 

We construct tuples $p_0,\dots,p_{k-1}\in\{0,1\}^n$ as follows. Index the coordinates by the elements of $\mathbb Z_n$. 
For each $i\in\{0,\dots,k-1\}$, let $p_i$ have ones precisely in the positions
$$bi, bi+1,\dots, bi+(b-1) \pmod n.$$
Thus each tuple contains exactly $b$ ones.

For every coordinate, the number of tuples containing a $1$ in that coordinate is at most $\lceil b\cdot \frac{k}{n}\rceil$. Since  
$$\left\lceil b\cdot \frac{k}{n}\right\rceil \leq \left\lceil\beta(n,k)\cdot \frac{k}{n}\right\rceil = \left\lceil\left\lfloor\frac{k-1}{k}\cdot n\right\rfloor \cdot \frac{k}{n}\right\rceil \leq (k-1)$$
Therefore no coordinate contains a $1$ in all $k$ tuples. This contradicts the assumption that $f$ is compatible with $\relation{R}_k$.
\end{proof}

\begin{remark}
The above lemma applies to function from $\Bc_k$ and $\Bc_k^\leq$.
\end{remark}

\begin{lemma}
Let $k,n\in \mathbb{N}$ and let $b > \beta(n,k)$. Define $\atleast{b}$ be the symmetric $n$-ary function that returns $1$ if and only if the input contains at least $b$ ones. Then $\atleast{b}\in B_k$ and $\atleast{b}\in B^\leq_k$.
\end{lemma}
\begin{proof}
Consider any $k$ tuples that are mapped to $1$ by $\atleast{b}$. Since each of them contains at least $b$ ones, the total number of ones among these tuples is at least
$$k\cdot b >  k\cdot \left\lfloor\frac{k-1}{k}\cdot n\right\rfloor \geq (k-1)\cdot n$$
$$k\cdot b \geq k\cdot(\beta(n,k)+1) = k\cdot \left(\left\lfloor\frac{k-1}{k}\cdot n\right\rfloor + 1\right) > k\cdot \left(\frac{k-1}{k}\cdot n\right) \geq (k-1)\cdot n.$$
By the pigeonhole principle, there must be a coordinate containing $1$ in all $k$ tuples. Hence $\atleast{b}$ is compatible with $\relation{R}_k$, and therefore $\atleast{b}\in\Bc_k$.
Since $\atleast{b}$ is monotone by definition, we also have $\atleast{b}\in B_k^\leq$.
\end{proof}

\subsection{Computation of cores}

We begin with several lemmas that will be useful for constructing minion homomorphisms.

\begin{definition}[dual function] 
Let $f$ be a Boolean function. The {\em dual} of $f$ is the function $f^*$ defined by $f^*(a) = \overline{f(\bar{a})}$.
\end{definition}

\begin{lemma}
  \label{binary-minion-homomorphism}
  Let $\clone{P} = \C{X_1,X_2,\dots,X_n}$ be a product of clones, let $\rm bin\colon \{0,1\}^2\to\{0,1\}$ be a binary function, and let $\alpha_1,\alpha_2$ be minion homomorphisms from $\clone{P}$ to a minion $\minion{Y}$.
  Define
  $$\alpha(f)(x) ={\rm bin}(\alpha_1(f)(x), \alpha_2(f)(x))$$
  for every $f\in\clone{P}$
  Then $\alpha$ is a minion homomorphism from $P$ to $Y$.
\end{lemma}
\begin{proof}
  Fix $m,n\in\mathbb N$, indices $i_1,i_2,\dots,i_n\in[m]$ and an $n$-ary function $f\in\clone{P}$. Since $\alpha_1$ and $\alpha_2$ are minion homomorphisms, we have
\begin{align*}
  \alpha_1(f)\circ(\pi_{i_1}^m,\pi_{i_2}^m,\dots,\pi_{i_n}^m) &= \alpha_1(f\circ(\pi_{i_1}^m,\pi_{i_2}^m,\dots,\pi_{i_n}^m)) \\
  \alpha_2(f)\circ(\pi_{i_1}^m,\pi_{i_2}^m,\dots,\pi_{i_n}^m) &= \alpha_2(f\circ(\pi_{i_1}^m,\pi_{i_2}^m,\dots,\pi_{i_n}^m)) \\
\end{align*}
Therefore
\begin{align*}
  \alpha(f)\circ(\pi_{i_1}^m,\pi_{i_2}^m,\dots,\pi_{i_n}^m) &= \\
        {\rm bin}\big(\alpha_1(f),\alpha_2(f))\circ( \pi_{i_1}^m,\pi_{i_2}^m,\dots,\pi_{i_n}^m\big) &= \\
        {\rm bin}\big(\alpha_1(f) \circ( \pi_{i_1}^m,\pi_{i_2}^m,\dots,\pi_{i_n}^m),\alpha_2(f) \circ( \pi_{i_1}^m,\pi_{i_2}^m,\dots,\pi_{i_n}^m)\big) &= \\
        {\rm bin}\big(\alpha_1(f \circ( \pi_{i_1}^m,\pi_{i_2}^m,\dots,\pi_{i_n}^m)),\alpha_2(f \circ( \pi_{i_1}^m,\pi_{i_2}^m,\dots,\pi_{i_n}^m))\big) &= \\
        \alpha(f \circ(\pi_{i_1}^m,\pi_{i_2}^m,\dots,\pi_{i_n}^m)) &= \\
\end{align*}
\end{proof}

\begin{lemma}
  \label{dual-minion-homomorphism}
Let $\clone{X}$ be a clone. Then the map given by $f\mapsto f^*$ is a minion homomorphism.
\end{lemma}
\begin{proof}
As in the proof of Lemma~\ref{binary-minion-homomorphism}, this follows directly from the definitions of a minion homomorphism and of the dual function.
\end{proof}

\begin{lemma}
\label{bi-mensi}
  Let $b_i\in \Bc_i$, and let $f$ be any function of the same arity, then $b_i\wedge f \in \Bc_i$. 
\end{lemma}
\begin{proof}
Since $(b_i\wedge f)^{-1}(1)\subseteq b_i^{-1}(1),$, it follows immediately from the definition of $\Bc_i$ that $b_i\wedge f\in\Bc_i.$
\end{proof}

\begin{lemma}
\label{bi-less-mensi}
  Let $b_i\in \Bc_i^\leq$, and let $f$ be any monotone function of the same arity. Then $b_i\wedge f \in \Bc_i^\leq$. 
\end{lemma}
\begin{proof}
As in the proof of Lemma~\ref{bi-mensi}, we have $(b_i\wedge f)\in \Bc_i.$ It therefore remains to show that $b_i\wedge f$ is monotone. Since both $b_i$ and $f$ are monotone, their pointwise conjunction $b_i\wedge f$ is also monotone. 
\end{proof}

\begin{lemma}
  \label{c-mensi}
Let $c\in \Cc$ and $a\in \Ac$, then $((c\vee a)\wedge a^*)\in \Cc$.
\end{lemma}
\begin{proof}
  Let $n$ the common arity of $c$ and $a$, and let $x\in \{0,1\}^n$. Since $c\in\Cc$, exactly one of $c(x)$ or $c(\bar{x})$ is equal to $1$. Without loss of generality, assume that $c(x)=1$ and therefore $c(\bar{x})=0$. We distinguish the following cases.
  \begin{itemize}
  \item If $a(x) = 1$ and $a(\bar{x}) = 0$, then $((c\vee a)\wedge a^*) (x) = 1$ and $((c\vee a)\wedge a^*) (\bar{x}) = 0$
  \item If $a(x) = 0$ and $a(\bar{x}) = 1$, then $((c\vee a)\wedge a^*) (x) = 0$ and $((c\vee a)\wedge a^*) (\bar{x}) = 1$
  \item If $a(x) = 0$ and $a(\bar{x}) = 0$, then $((c\vee a)\wedge a^*) (x) = c(x)$ and $((c\vee a)\wedge a^*) (\bar{x}) = c(\bar{x})$
  \end{itemize}
  Hence $((c\vee a)\wedge a^*)$ preserves the relation $\neq$, and therefore $(c\vee a)\wedge a^* \in \Cc$.
\end{proof}

Our strategy for proving that the core of a product clone $\clone{P}$ is a minion $\minion{G}$ is as follows. To establish the existence of a minion homomorphism $\clone{P} \to \minion{G}$, we provide an explicit formula. The preceding lemmas guarantee that this formula indeed defines a minion homomorphism. In every case, the homomorphism $\minion{G} \to \clone{P}$ is simply the inclusion map.

It therefore remains to prove that $\minion{G}$ is a minion core. To this end, we apply Lemma~\ref{proof-minion}. By that lemma, it suffices to show that every binary function in $\minion{G}$ belongs to the image of every minion endomorphism. 

For each case, we consider a particular set of height-one identities $(*)$ involving a  function symbol $f$. We then establish two facts. First, any function satisfying $(*)$ has, among its minors, all non-projection binary functions of $\minion{G}$. Second, there exists a function $f\in\minion{G}$ satisfying $(*)$.

Now let $\alpha$ be a minion endomorphism of $\minion{G}$. Since minion homomorphisms preserve height-one identities, the function $\alpha(f)$ also satisfies $(*)$. Therefore, the minors of $\alpha(f)$ contain all non-projection binary functions of $\minion{G}$. Since, by Lemma \ref{projtoproj}, projections are also in the image of $\alpha$,
it follows that every binary function of $\minion{G}$ belongs to the image of $\alpha$.

Hence every binary function in $\minion{G}$ lies in the image of every minion endomorphism of $\minion{G}$. By Lemma~\ref{proof-minion}, we conclude that $\minion{G}$ is a minion core.

\begin{lemma}\label{lcleqbjbi} $\C{\Lc,\Cc^\leq,\Bc_j^\leq,\Bc_i}$, where $2<j<i$, has core $\minion{G}$ consisting of all quadruples $(\ell,c,b_j,b_i)$ satisfying $b_i \leq b_j \leq c, b_i\leq \ell$. 
\end{lemma}

\begin{proof}

A homomorphism from the product to $\minion{G}$ is given by $(\ell,c,b_j,b_i)\mapsto (\ell, c, b_j\wedge c, b_i\wedge b_j\wedge c\wedge \ell)$, which is a minion homomorphism by Lemmas \ref{binary-minion-homomorphism}, \ref{bi-mensi}, and \ref{bi-less-mensi}.

Let $n\equiv 1\pmod 4$ be sufficiently large so that $\beta(n,i) - \beta(n,j) > 3$ and $\beta(n,j)-\lceil\frac{n}{2}\rceil > 1$. Choose an even integer $b$ satisfying $ \beta(n,j) < b <\beta(n,i)-1$.

Consider $n$-ary functions $f = (\ell,c,b_j,b_i)\in\minion{G}$ satisfying the following identities.
\begin{enumerate}
    \item $f$ is $(\lceil n/2\rceil, n-1, b, b+1)$-pseudosymmetric.
    \item $f(x_{n-2}) = x$.
    \item $f(x_{n-1}) = f(x_{b})$.
\end{enumerate}

From $f$ being $(\lceil n/2\rceil)$-pseudosymmetric, we obtain that $\ell$ is the sum of all variables and $c$ is the majority on all variables. From (2) we obtain $b_i(1_{n-2}) = b_j(1_{n-2}) = 1$. By monotonicity $b_j(1_{n-1}) = 1$, and therefore, using (3), $b_j(1_b)=1$. Applying monotonicity once more yields $b_j(1_{b+1}) = 1$. Since $b+1 < \beta(n,i)$, we have $b_i(1_{b+1}) = b_i(1_{b}) = 0$.

The known values are summarized in the following table.

\begin{center}
\begin{tabular}{ |cccc|l } 
 \hline
 $\ell$ & $c$ & $b_j$ & $b_i$ & number of ones \\ 
  \hline
     $1$ & $1$ & $1$ & $1$ & $n$ \\ 
     $0$ & $1$ & $?$ & $?$ & $n-1$, $b_j(1_{n-1}) = 1$ \\ 
     $1$ & $1$ & $?$ & $?$ & $n-2$, $b_i(1_{n-2}) = b_j(1_{n-2}) = 1$ \\
\multicolumn{4}{|c|}{$\cdots$} & \color{gray}$\longleftarrow\beta(n,i)$ \\
     $1$ & $1$ & $1$ & $0$ &  $b+1$ \\
     $0$ & $1$ & $1$ & $0$ &  $b$ \\
\multicolumn{4}{|c|}{$\cdots$}  & \color{gray}$\longleftarrow\beta(n,j)$\\
     $0$ & $1$ & $0$ & $0$ & $\lceil\frac{n}{2}\rceil +1$ \\
     $1$ & $1$ & $0$ & $0$ & $\lceil\frac{n}{2}\rceil$ \\
     $0$ & $0$ & $0$ & $0$ & $\lfloor\frac{n}{2}\rfloor$ \\
\multicolumn{4}{|c|}{$\cdots$} \\
     $0$ & $0$ & $0$ & $0$ & even \\
     $1$ & $0$ & $0$ & $0$ & odd \\
\multicolumn{4}{|c|}{$\cdots$} \\
     $0$ & $0$ & $0$ & $0$ & $0$\\
 \hline
\end{tabular}
\end{center}

\begin{note}
If $i=\infty$, then the only change is that $\beta(n,i) = n-1$; the function values are otherwise unchanged.
\end{note}

Using the table, we obtain  
\begin{align*}
f(x_{b+1}) &= (xxx{\wedge}) \\
f(x_b) &= (yxx{\wedge}) \\
f(x_{\lceil n/2 \rceil +1}) &= (yx{\wedge}{\wedge}) \\
f(x_{\lceil n/2 \rceil}) &= (xx{\wedge}{\wedge})
\end{align*}
The right-hand sides include all non-projection binary functions in $\minion{G}$ up to swapping $x$ and $y$.
(Note that the other binary functions are are also minors, e.g., $f(y_b) = (xyy{\wedge})$).

Finally, there exists $f\in\C{\Lc,\Cc^\leq,\Bc_j^\leq,\Bc_i}\in\minion{G}$ satisfying these identities; for example $f=(\ell,c,b_j,b_i)$, where
\begin{itemize}
\item $\ell$ is the sum of all variables.
\item $c$ is majority on all variables.
\item $b_j$ is $\atleast{b}$.
\item $b_i$ is such that $b_i(1_{n-2}) = b_i(1_n) = 1$ and zero everywhere
else.
\end{itemize}
\end{proof}

    Using exactly the same identities as in the previous lemma, we obtain the following corollary.

    \begin{corollary}
      \label{lcleqbjbicor}
\hfill
    \begin{itemize}
        \item $\C{\Lc}$ is a core.
        \item $\C{\Cc^\leq}$ is a core.
        \item $\C{\Bc_j^\leq}$, where $2<j<\infty$, is a core.
        \item $\C{\Bc_i}$, where $3<i$, is a core.
        \item $\C{\Lc,\Cc^\leq}$ is a core.
        \item $\C{\Lc,\Bc_j^\leq}$, where $2<j<\infty$, is a core.
        \item $\C{\Lc, \Bc_i}$, where $3<i$, has core $\minion{G}$ consisting of all pairs $(\ell,b_i)$  satisfying $b_i\leq \ell$. 
        \item $\C{\Cc^\leq,\Bc_j^\leq}$, where $2<j<\infty$, has core $\minion{G}$ consisting of all pairs of functions $c, b_j$ satisfying $b_j\leq c$.
        \item $\C{\Cc^\leq, \Bc_i}$, where $3<i$, has core $\minion{G}$ consisting of all pairs of functions $c, b_i$ satisfying $b_i\leq c$.
        \item $\C{\Bc_j^\leq, \Bc_i}$, where $2<j<i$, has core $\minion{G}$ consisting of all pairs of functions $b_j, b_i$ satisfying $b_i\leq b_j$.
        \item $\C{\Lc,\Cc^\leq,\Bc_j^\leq}$, where $2<j<\infty$, has core $\minion{G}$ consisting of all triples of functions $\ell,c, b_j$ satisfying $b_j\leq c$.
        \item $\C{\Lc,\Cc^\leq,\Bc_i}$, where $3<i$, has core $\minion{G}$ consisting of all triples of functions $\ell,c, b_i$ satisfying $b_i\leq c, b_i\leq \ell$.
        \item $\C{\Lc,\Bc_j^\leq,\Bc_i}$, where $2<j<i$, has core $\minion{G}$ consisting of all triples of functions $\ell,b_j, b_i$ satisfying $b_i\leq b_j, b_i\leq \ell$.
        \item $\C{\Cc^\leq, \Bc_j^\leq,\Bc_i}$, where $2<j<i$, has core $\minion{G}$ consisting of all triples of functions $c,b_j, b_i$ satisfying $b_i\leq b_j\leq c$.
    \end{itemize}
\end{corollary}

\begin{lemma}\label{lcleqbleqinf} $\C{\Lc,\Cc^\leq,\Bc_\infty^\leq}$, has core $\minion{G}$ consisting of all triples of functions $(\ell,c,b_\infty)$ satisfying $b_\infty \leq c$.
\end{lemma}
\begin{proof}
  A homomorphism from the product to $\minion{G}$ is given by $(\ell,c,b_\infty)\mapsto (\ell,c,b_{\infty}\wedge c)$, which is a minion homomorphism
by Lemmas \ref{binary-minion-homomorphism}, \ref{bi-mensi}, and \ref{bi-less-mensi}.

Consider $9$-ary functions $f = (\ell,c,b_\infty)$ satisfying identities
\begin{enumerate}
    \item $f$ is $(5,6)$-pseudosymmetric.
    \item $f(x_{7}) = x$.
\end{enumerate}

From $f$ being $(5)$-pseudosymmetric, we obtain that $\ell$ is the sum of all variables and $c$ is the majority on all variables. From (2) and the monotonicity of $b_\infty $we get $b_\infty(1_8) = 1$. From (1) and monotonicity, we have $b_\infty(0_k) = 0$ for $k\leq 6$.

\begin{center}
\begin{tabular}{ |ccc|l } 
 \hline
 $\ell$ & $c$ & $b_\infty$  & number of ones \\ 
  \hline
$1$ & $1$ & $1$ & $9$ \\
$0$ & $1$ & ? & $8$, $b_\infty(1_8) = 1$ \\
$1$ & $1$ & ? & $7$, $b_\infty(1_7) = 1$\\
$0$ & $1$ & $0$ & $6$ \\
$1$ & $1$ & $0$ & $5$ \\ 
$0$ & $0$ & $0$ & $4$ \\
$1$ & $0$ & $0$ & $3$ \\ 
$0$ & $0$ & $0$ & $2$ \\
$1$ & $0$ & $0$ & $1$ \\ 
$0$ & $0$ & $0$ & $0$ \\
 \hline
\end{tabular}
\end{center}

All non-projection binary functions in $\minion{G}$ are minors of $f$:
\begin{align*}
f(x_8) &= (yxx) \\
f(x_6) &= (yx{\wedge}) \\
f(x_{5}) &= (xx{\wedge}) \\
\end{align*}

Finally, there exists $f\in\C{\Lc,\Cc^\leq,\Bc_\infty^\leq}$ satisfying these identities; for example $f=(\ell,c,b_\infty)$, where
\begin{itemize}
\item $\ell$ is the sum of all variables.
\item $c$ is majority on all variables.
\item $b_\infty$ is such that $b_\infty(1_9)=b_\infty(1_8)=b_\infty(1_7)=1$ and zero everywhere else.
\end{itemize}
\end{proof}

Using exactly the same identities as in the previous lemma, we obtain
the following corollary.

\begin{corollary}\label{lcleqbleqinfcor}
  \hfill
    \begin{itemize}
        \item $\C{\Bc^\leq_\infty}$ is a core.
        \item $\C{\Lc,\Bc^\leq_\infty}$ is a core.
        \item $\C{\Cc^\leq,\Bc^\leq_\infty}$, has core $\minion{G}$ consisting of all pairs of functions $c, b_\infty$ satisfying $b_\infty\leq c$.
    \end{itemize}
\end{corollary}

\begin{lemma}\label{lb2leqbi} $\C{\Lc,\Bc_2^\leq, \Bc_i}$, where $2<i$, has core $\minion{G}$ consisting of triples of functions $(\ell,b_2,b_i)$ satisfying $b_i \leq b_2 \leq \ell$.
\end{lemma}
\begin{proof}
  A homomorphism from the product to $\minion{G}$ is given by $(\ell, b_2, b_i)\mapsto (\ell, b_2, b_i\wedge b_2\wedge \ell)$, which is a minion homomorphism
by Lemmas \ref{binary-minion-homomorphism} and \ref{bi-mensi}.

Let $n\equiv 1\pmod 4$ be sufficiently large so that $\beta(n,i)-\left\lceil\frac{n}{2}\right\rceil > 6$. 
Consider $n$-ary functions $f = (\ell,b_2,b_i)$ satisfying the identities
\begin{enumerate}
    \item $f$ is $(\left\lceil\frac{n}{2}\right\rceil, \left\lceil\frac{n}{2}\right\rceil+1, \left\lceil\frac{n}{2}\right\rceil+3, \left\lceil\frac{n}{2}\right\rceil+4)$-pseudosymmetric.
    \item $f(x_{\left\lceil\frac{n}{2}\right\rceil+2}) = x$.
    \item $f(x_{\left\lceil\frac{n}{2}\right\rceil+1}) = f(x_{\left\lfloor\frac{n}{2}\right\rfloor})$
\end{enumerate}

From $f$ being $(\left\lceil\frac{n}{2}\right\rceil)$-pseudosymmetric we obtain that $\ell$ is the sum of all variables. From (2) and the monotonicity of $b_2$, we get $b_2(1_{\left\lceil\frac{n}{2}\right\rceil +3}) = b_2(1_{\left\lceil\frac{n}{2}\right\rceil +4}) = 1$. Using (3), monotonicity of $b_2$ and the fact that $\lfloor \frac{n}{2}\rfloor = \beta(n,2)$ we obtain $b_2(1_{\left\lceil\frac{n}{2}\right\rceil +1}) = b_2(1_{\left\lceil\frac{n}{2}\right\rceil}) = 0$.

\begin{center}
\begin{tabular}{ |ccc|l } 
 \hline
 $\ell$ & $b_2$ & $b_i$ & number of ones \\ 
  \hline
     $1$ & $1$ & $1$ & $n$ \\ 
\multicolumn{3}{|c|}{$\cdots$} & \color{gray}$\longleftarrow\beta(n,i)$ \\
     $1$ & $1$ & $0$  &  $\lceil\frac{n}{2}\rceil +4$ \\
     $0$ & $1$ & $0$  &  $\lceil\frac{n}{2}\rceil +3$ \\
     $1$ & ? & ?  &  $\lceil\frac{n}{2}\rceil +2$, $b_i(1_{\lceil\frac{n}{2}\rceil +2}) = 1$ \\
     $0$ & $0$ & $0$ & $\lceil\frac{n}{2}\rceil +1$ \\
     $1$ & $0$ & $0$  & $\lceil\frac{n}{2}\rceil$ \\
     $0$ & $0$ & $0$  & $\lfloor\frac{n}{2}\rfloor$ \\
\multicolumn{3}{|c|}{$\cdots$} \\
     $0$ & $0$ & $0$  & even \\
     $1$ & $0$ & $0$  & odd \\
\multicolumn{3}{|c|}{$\cdots$} \\
     $0$ & $0$ & $0$  & $0$\\
 \hline
\end{tabular}
\end{center}

Using the table, we obtain
\begin{align*}
f(x_{\lceil\frac{n}{2}\rceil +4}) &= (xx{\wedge}) \\
f(x_{\lceil\frac{n}{2}\rceil +3}) &= (yx{\wedge}) \\
f(x_{\lceil\frac{n}{2}\rceil +1}) &= (y{\wedge}{\wedge}) \\
\end{align*}

The right-hand sides include all non-projection binary functions in $\minion{G}$.

Finally, there exists $f\in\C{\Lc,\Bc_2^\leq, \Bc_i}$ satisfying these identities, for example $f=(\ell,b_2,b_i)$, where
\begin{itemize}
\item $\ell$ is the sum of all variables.
\item $b_2$ is $\atleast{\lceil\frac{n}{2}\rceil+2}$
\item $b_i$ is such that $b_i({1_{\lceil\frac{n}{2}\rceil +2}})=b_i(1_n)=1$ and zero everywhere else.
\end{itemize}
\end{proof}

Using the exactly same equalities as in the previous lemma, we obtain the following corollary.

\begin{corollary}\label{lb2leqbicor}
  \hfill
    \begin{itemize}
        \item $\C{\Bc_2^\leq}$ is a core.  
        \item $\C{\Lc, \Bc_2^\leq}$ is a core.
        \item $\C{\Lc, \Bc_3}$, has core $\minion{G}$ consisting of pairs of functions $(\ell, b_3)$ satisfying $b_3\leq \ell$.
        \item $\C{\Bc_2^\leq, \Bc_i}$, where $3<i$, has core $\minion{G}$ consisting of pairs of functions $b_2, b_i$ satisfying $b_i\leq b_2$.
    \end{itemize}
\end{corollary}

\begin{lemma}\label{cbjbi} $\C{\Cc,\Bc_j^\leq, \Bc_i}$, where $1<j<i$, has core $\minion{G}$ consisting of triples of functions $(c,b_j,b_i)$ satisfying $b_i \leq b_j\leq c $.
\end{lemma}
\begin{proof}
  A homomorphism from the product to $\minion{G}$ is provided by the map $(c,b_j,b_i)\mapsto (c, b_j\wedge c, b_i\wedge b_j\wedge c)$, which is a minion homomorphism
by Lemmas \ref{binary-minion-homomorphism}, \ref{bi-mensi}, and \ref{bi-less-mensi}

Let us take odd $n$ sufficiently large so that $\beta(n,i)-\beta(n,j) > 1$,  $\beta(n,i)-\left\lceil\frac{n}{2}\right\rceil > 1$. 
We consider $n$-ary functions $f = (c,b_j,b_i)$ satisfying identities
\begin{enumerate}
    \item $f$ is $(\lceil\frac{n}{2}\rceil, \beta(n,j)+2)$-pseudosymmetric.
    \item $f(x_{\beta(n,j)+1}) = x$.
    \item $f(x_{\left\lceil\frac{n}{2}\right\rceil +1}) = f(x_{\left\lfloor\frac{n}{2}\right\rfloor})$
\end{enumerate}

From (1), (2) and $b_j$ monotone we have that $b_j(1_{\beta(n,j)+2}) = 1$. From (2), (3) and the monotonicity of $b_j$ we have that $b_j(1_{\left\lceil\frac{n}{2}\right\rceil +1}) = b_j(1_{\left\lceil\frac{n}{2}\right\rceil}) = 0$. Let $p = c(1_{\lceil\frac{n}{2}\rceil})$ and $q = c(1_{\lceil\frac{n}{2}\rceil +1})$

\begin{center}
\begin{tabular}{ |ccc|l } 
 \hline
 $c$ & $b_j$ & $b_i$ & number of ones \\ 
  \hline
     $1$ & $1$ & $1$ & $n$ \\ 
\multicolumn{3}{|c|}{$\cdots$} & \color{gray}$\longleftarrow\beta(n,i)$ \\
     $1$ & $1$ & $0$  &  $\beta(n,j)+2$ \\
    ? & ? & ?  &  $\beta(n,j)+1$, $b_j(1_{\beta(n,j)+1}) = 1$ \\
\multicolumn{3}{|c|}{$\cdots$} & \color{gray}$\longleftarrow\beta(n,j)$ \\
     $q$ & $0$ & $0$ & $\lceil\frac{n}{2}\rceil +1$ \\
     $p$ & $0$ & $0$  & $\lceil\frac{n}{2}\rceil$ \\
     $\bar{p}$ & $0$ & $0$  & $\lfloor\frac{n}{2}\rfloor$ \\
    $\bar{q}$ & $0$ & $0$  & $\lfloor\frac{n}{2}\rfloor-1$ \\
\multicolumn{3}{|c|}{$\cdots$} \\
     $0$ & $0$ & $0$  & $0$\\
 \hline
\end{tabular}
\end{center}

Using this table, we obtain  
\begin{align*}
f(x_{\beta(n,j)+2}) &= (xx{\wedge}) \\
f(x_{\lceil\frac{n}{2}\rceil}) &= \left\{\begin{matrix}(x{\wedge}{\wedge}) \\ (y{\wedge}{\wedge})\end{matrix}\right.
\end{align*}

The right-hand sides include all non-projection binary functions in $\minion{G}$. 

Finally, there exists $f\in\C{\Cc,\Bc_j^\leq, \Bc_i}$ satisfying these identities, for example $f=(c,b_j,b_i)$, where
\begin{itemize}
\item $c$ is majority on all elements with the exception that $c(1_{\left\lceil\frac{n}{2}\right\rceil +1}) = 0$ and $c(0_{\left\lceil\frac{n}{2}\right\rceil +1}) = 1$.
\item $b_j$ is $\atleast{\beta(n,j)+1}$
\item $b_i$ is such that $b_i({1_{\beta(n,j)+1}})=b_i(1_n)=1$ and zero everywhere else.
\end{itemize}
\end{proof}

Using exactly the same identities as in the previous lemma, we obtain
the following corollary.

\begin{corollary}\label{cbjbicor}
  \hfill
    \begin{itemize}
        \item $\C{\Cc}$ is a core.
        \item $\C{\Bc_3}$ is a core.
        \item $\C{\Cc, \Bc_j^\leq}$, where $1<j<\infty$, has core $\minion{G}$ consisting of pairs of functions $c, b_j$ satisfying $b_j\leq c$.
        \item $\C{\Cc, \Bc_i}$, where $2<i$, has core $\minion{G}$ consisting of pairs of functions $c, b_i$ satisfying $b_i\leq c$.
    \end{itemize}
\end{corollary}

\begin{lemma}\label{lb2} $\C{\Lc, \Bc_2}$, has core $\minion{G}$ consisting of pairs of functions $(\ell,b_2)$ satisfying $b_2\leq \ell $.
\end{lemma}
\begin{proof}
  A homomorphism from the product to $\minion{G}$ is provided by the map $(\ell,b_2)\mapsto (\ell, b_2\wedge \ell)$, which is a minion homomorphism by Lemmas \ref{binary-minion-homomorphism} and Lemma \ref{bi-mensi}.

Consider $3$-ary functions $f = (\ell, b_2)$, which are fully symmetric

From the symmetry condition we have that $\ell$ is the sum of the three variables modulo $2$. Because $b_2\leq \ell$ we have that $b_2(1_2) = \ell(1_2)=0$. Further more we get $b_2(1_1) = 0$, because $1= \beta(3,2)$.

\begin{center}
\begin{tabular}{ |cc|l } 
 \hline
 $\ell$ & $b_2$ & number of ones \\ 
  \hline
     $1$ & $1$ & $3$ \\ 
     $0$ & $0$ & $2$ \\
     $1$ & $0$ & $1$ \\
     $0$ & $0$ & $0$ \\    
 \hline
\end{tabular}
\end{center}

Using this table, we may see  $f(x_1) = (x{\wedge})$, which is the only nontrivial binary minor in $\minion{G}$. 

Finally, there exists $f\in\C{\Lc, \Bc_2}$ satisfying these identities, for example $f=(\ell,b_2)$, where
\begin{itemize}
\item $\ell$ is sum of the three variables modulo $2$.
\item $b_2$ is $\atleast{3}$.
\end{itemize}
\end{proof}

Using exactly the same identities as in the previous lemma, we obtain
the following corollary.

\begin{corollary}\label{lb2cor}
  \hfill
    \begin{itemize}
        \item $\C{\Bc_2}$ is a core.
    \end{itemize}
\end{corollary}

\begin{lemma}\label{lcleqb3} $\C{\Lc,\Cc^\leq,\Bc_3}$ has core $\minion{G}$ consisting of triples of functions $(\ell,c,b_3)$ satisfying $b_3 \leq \ell$.
\end{lemma}
\begin{proof}
  A homomorphism from the product to $\minion{G}$ is provided by the map $(\ell,c,b_3)\mapsto (\ell, c, b_3\wedge c\wedge \ell)$, which is a minion homomorphism by Lemmas \ref{binary-minion-homomorphism} and \ref{bi-mensi}.

Let us take $n\equiv 1\pmod 4$ sufficiently large so that $\beta(n,3) - \lfloor\frac{n}{2}\rfloor \geq 2$.

Consider $n$-ary functions $f= (\ell, c, b_3)$  that are fully symmetric.
From $f$ being symmetric we have that $\ell$ is sum of all variables and $c$ is majority of all variables. Because $\beta(n,3) - \lfloor\frac{n}{2}\rfloor \geq 2$ we have for all $i\leq \lceil\frac{n}{2}\rceil+1$ that $b_3(1_i)= 0$.

\begin{center}
\begin{tabular}{ |ccc|l } 
 \hline
 $\ell$ & $c$ & $b_3$ & number of ones \\ 
  \hline
     $1$ & $1$ & $1$ & $n$ \\ 
\multicolumn{3}{|c|}{$\cdots$} & \color{gray}$\longleftarrow\beta(n,3)$ \\
     $0$ & $1$ & $0$ & $\lceil\frac{n}{2}\rceil +1$ \\
     $1$ & $1$ & $0$ & $\lceil\frac{n}{2}\rceil$ \\
     $0$ & $0$ & $0$ & $\lfloor\frac{n}{2}\rfloor$ \\
     $1$ & $0$ & $0$ & $\lfloor\frac{n}{2}\rfloor -1$ \\ 
\multicolumn{3}{|c|}{$\cdots$} \\
     $0$ & $0$ & $0$ & even \\
     $1$ & $0$ & $0$ & odd \\
\multicolumn{3}{|c|}{$\cdots$} \\
     $0$ & $0$ & $0$ & $0$\\
 \hline
\end{tabular}
\end{center}

Using this table, we may see  
\begin{align*}
f(x_{\lceil n/2 \rceil +1}) &= (yx{\wedge}) \\
f(x_{\lceil n/2 \rceil}) &= (xx{\wedge})
\end{align*}

The right-hand sides include all non-projection binary functions in $\minion{G}$.

There exists $f\in\C{\Lc,\Cc^\leq,\Bc_3}$ satisfying these identities, for example $f=(\ell,c,b_3)$, where
\begin{itemize}
\item $\ell$ is sum of the three variables modulo $2$.
\item $c$ is majority on the three variables.
\item $b_3$ is $\atleast{n}$.
\end{itemize}
\end{proof}

\begin{lemma}\label{cb2} $\C{\Cc,\Bc_2}$ has core $\minion{G}$ consisting of pairs of functions $(c,b_2)$ satisfying $b_2 \leq c$.
\end{lemma}
\begin{proof}
  A homomorphism from the product to $\minion{G}$ is provided by the map $(c,b_2)\mapsto (c, b_2\wedge c)$, which is a minion homomorphism by Lemmas \ref{binary-minion-homomorphism} and \ref{bi-mensi}.

We consider $5$-ary functions $f= (c, b_2)$ that satisfies the following identities
\begin{itemize}
    \item $f$ is fully symmetric
    \item $f(x_3) = f(x_1)$
\end{itemize}
Because $2\leq \beta(5,2)$ we get that $b_2(1_1) = b_2(1_2) = 0$ and using (2) we get $b_2(1_3) = b_2(1_4) = 0$.

\begin{center}
\begin{tabular}{ |cc|l } 
 \hline
 $c$ & $b_2$ & number of ones \\ 
  \hline
     $1$ & $1$ & $5$ \\ 
     $\bar{p}$ & $0$ & $4$ \\ 
     $p$ & $0$ & $3$ \\ 
     $\bar{p}$ & $0$ & $2$ \\ 
     $p$ & $0$ & $1$ \\ 
     $0$ & $0$ & $0$ \\ 
 \hline
\end{tabular}
\end{center}

Using this table, we obtain
$$f(x_2) = \left\{\begin{matrix}(x{\wedge}) \\ (y{\wedge})\end{matrix}\right.$$
The right-hand sides include all non-projection binary functions in $\minion{G}$.

Finally, there exists $f\in\C{C,B_2}$ satisfying these identities, for example $f=(c,b_2)$, where
\begin{itemize}
\item $c$ returns the number of ones in the input modulo 2.
\item $b_2$ is $\atleast{5}$.
\end{itemize}
\end{proof}

\begin{lemma}\label{cbleqinf} $\C{\Cc,\Bc_\infty^\leq}$ has core $\minion{G}$ consisting of pairs of functions $(c,b_\infty)$ satisfying $b_\infty \leq c$.
\end{lemma}
\begin{proof}
  A homomorphism from the product to $\minion{G}$ is provided by the map $(c,b_\infty) \mapsto (c,b_\infty \wedge c)$, which is a minion homomorphism by Lemmas \ref{binary-minion-homomorphism} and \ref{bi-less-mensi}. 

We consider $3$-ary fully symmetric functions $f= (c, b_\infty)$. 
Because $2 =  \beta(3,\infty)$ we get that $b_2(1_1) = b_2(1_2) = 0$. 

\begin{center}
\begin{tabular}{ |cc|l } 
 \hline
 $c$ & $b_\infty$ & number of ones \\ 
  \hline
     $1$ & $1$ & $3$ \\ 
     $\bar{p}$ & $0$ & $2$ \\ 
     $p$ & $0$ & $1$ \\ 
     $0$ & $0$ & $0$ \\ 
 \hline
\end{tabular}
\end{center}

Using this table, we obtain
$$f(x_2) = \left\{\begin{matrix}(x{\wedge}) \\ (y{\wedge})\end{matrix}\right.$$
The right-hand sides include all non-projection binary functions in $\minion{G}$.

Finally, there exists $f\in\C{\Cc,\Bc_\infty^\leq}$ satisfying these identities, for example $f=(c,b_\infty)$, where
\begin{itemize}
\item $c$ returns the number of ones in the input modulo 2.
\item $b_\infty$ is $\atleast{3}$.
\end{itemize}
\end{proof}

\begin{lemma}\label{cleqb3} $\C{\Cc^\leq,\Bc_3}$ has core $\minion{G}$ consisting of pairs of functions $(c,b_3)$ satisfying $b_3 \leq c$.
\end{lemma}
\begin{proof}
  A homomorphism from the product to $\minion{G}$ is provided by the map $(c,b_3)\mapsto (c,b_3\wedge c)$, which is a minion homomorphism by Lemmas \ref{binary-minion-homomorphism} and \ref{bi-mensi}.

Let us take odd $n$  sufficiently large so that $\beta(n,3)-\lfloor\frac{n}{2}\rfloor\geq 1$.

Consider $n$-ary fully symmetric functions $f= (c, b_3)$.
Because $\beta(n,3) \geq \lceil\frac{n}{2}\rceil + 1$ we get that $b_3(1_{\lceil\frac{n}{2}\rceil}) = 0$. Because $c$ is fully symmetric, we have that $c$ is majority on all variables.

\begin{center}
\begin{tabular}{ |cc|l } 
 \hline
 $c$ & $b_3$ & number of ones \\ 
  \hline
     $1$ & $1$ & $n$ \\ 
     \multicolumn{2}{|c|}{$\cdots$} & \color{gray}$\longleftarrow\beta(n,3)$ \\
     $1$ & $0$ & $\lceil\frac{n}{2}\rceil$ \\
     $0$ & $0$ & $\lfloor\frac{n}{2}\rfloor$ \\
    \multicolumn{2}{|c|}{$\cdots$} \\
     $0$ & $0$ & $1$ \\ 
     $0$ & $0$ & $0$ \\ 
 \hline
\end{tabular}
\end{center}

Using this table, we obtain $f(x_{\lceil\frac{n}{2}\rceil}) = (x{\wedge})$
The right-hand sides include all non-projection binary functions in $\minion{G}$.

Finally, there exists $f\in\C{\Cc^\leq,\Bc_3}$ satisfying these identities, for example $f=(c,b_3)$, where
\begin{itemize}
\item $c$ is majority on all variables.
\item $b_3$ is $\atleast{n}$.
\end{itemize}
\end{proof}

\begin{lemma}\label{ca} $\C{\Cc,\Ac}$ has core $\minion{G}$ consisting of pairs of functions $(c,a)$ satisfying $a \leq c$.
\end{lemma}
\begin{proof}
  A homomorphism from the product to $\minion{G}$ is provided by the map $(c,a)\mapsto ((c\vee a)\wedge a^*,a)$, which is a minion homomorphism by Lemmas \ref{binary-minion-homomorphism} and \ref{c-mensi}.
  
We consider $3$-ary fully symmetric functions $f= (c, a)$. 
Because of the symmetry we have that $a$ is and of all variables.

\begin{center}
\begin{tabular}{ |cc|l } 
 \hline
 $c$ & $a$ & number of ones \\ 
  \hline
     $1$ & $1$ & $3$ \\ 
     $\bar{p}$ & $0$ & $2$ \\ 
     $p$ & $0$ & $1$ \\ 
     $0$ & $0$ & $0$ \\ 
 \hline
\end{tabular}
\end{center}

Using this table, we obtain
$$f(x_2) = \left\{\begin{matrix}(x{\wedge}) \\ (y{\wedge})\end{matrix}\right.$$
The right-hand sides include all non-projection binary functions in $\minion{G}$.

There exists $f\in\C{\Cc,\Ac}$ satisfying these identities, for example $f=(c,a)$, where
\begin{itemize}
\item $c$ is majority on all variables.
\item $a$ is and on all variables.
\end{itemize}
\end{proof}

Using exactly the same identities as in the previous lemma, we obtain
the following corollary.

\begin{corollary}\label{cacor}
  \hfill
    \begin{itemize}
        \item $\C{\Ac}$ is a core.
    \end{itemize}
\end{corollary}

Instead of just one function symbol $f$, the next lemma uses two function symbols, $f$ and $g$, each with its own height-one identities. The conclusion then follows in the same way: since the minors of $f$ and $g$ contain all nontrivial binary functions of $\minion{G}$, all of these functions also belong to the image of every minion endomorphism.

\begin{lemma}\label{lcleqa} $\C{\Lc,\Cc^\leq,\Ac}$ has core $\minion{G}$ consisting of triples of functions $(\ell,c,a)$ satisfying $a\leq c$.
\end{lemma}
\begin{proof}
  A homomorphism from the product to G is provided by the map $(\ell, c,a)\mapsto (\ell, (c\vee a)\wedge a^*,a)$, which is a minion homomorphism by Lemmas \ref{binary-minion-homomorphism} and \ref{c-mensi}.

Consider $5$-ary fully symmetric functions $f= (\ell, c, a)$. 
Because of the symmetry we have that $a$ is and of all variables, $c$ is majority on all variables and $\ell$ is sum of all variables.

\begin{center}
\begin{tabular}{ |ccc|l } 
 \hline
 $\ell$ & $c$ & $a$ & number of ones \\ 
  \hline
   $1$ & $1$ & $1$ & $5$ \\ 
   $0$ & $1$ & $0$ & $4$ \\ 
   $1$ &  $1$ & $0$ & $3$ \\ 
   $0$ &  $0$ & $0$ & $2$ \\ 
   $1$ &  $0$ & $0$ & $1$ \\ 
   $0$ &  $0$ & $0$ & $0$ \\ 
 \hline
\end{tabular}
\end{center}

Using this table, we obtain
\begin{align*}
f(x_{3}) &= (xx{\wedge}) \\
f(x_{4}) &= (yx{\wedge})
\end{align*}

The right-hand sides include all non-projection binary functions in $\minion{G}$ without the function $(yxx)$.

Consider $7$-ary functions $g=(\ell,c,a)$, such that they are block symmetric in the last $3$ variables and in the first $4$ variables. Additionally it satisfies $g(xxxx\ xyy) = g(xxxx\ xxx)$. Then $\ell(1111\ 100) = \ell(1111\ 111) = 1$. From the block symmetry we get that $\ell$ is either sum of all variables or sum of the last three variables. In either way we have $\ell(xxxx\ xxy) = x$. Then because $a$ is monotone we have $a(1111\ 110) \geq a(1111\ 100) = a(1111\ 111) = 1$ and because $c\geq a$ we have $c(1111\ 110) = 1$. 
Combining these we have that 
$$g(xxxx\ xxy) = (yxx)$$

The right side is the last non-projection binary function in $\minion{G}$.

Finally such $f$ and $g$ exist in $\minion{G}$. Example of such $f$ is already given by the table and an example $g = (\ell,c,a)$ is 
\begin{itemize}
\item $\ell$ is sum of all variables modulo $2$.
\item $c$ is majority on all variables.
\item $a$ is and on the first $4$ variables.
\end{itemize}
\end{proof}

Using exactly the same identities for $f$ and $g$ as in the previous lemma, we obtain
the following corollary.

\begin{corollary}\label{lcleqacor}
  \hfill
  \begin{itemize}
        \item $\C{\Lc, \Ac}$ is a core.
        \item $\C{\Cc^\leq, \Ac}$ is a core.
    \end{itemize}
\end{corollary}

\begin{lemma}\label{full} The clone $\clone{F}$ consisting of all functions has core $\minion{G}$ consisting of one constant function $c_r$ of each arity $r$, where $c_r$ always returns $0$.
\end{lemma}
\begin{proof}
A homomorphism from $\clone{F}$ to $\minion{G}$ is provided by the map $g\mapsto c_{{\rm ar}(g)}$, which is a minion homomorphism, since every minor of a constant function is again constant and therefore evaluates to $0$.

Consider binary functions $f\in\minion{G}$ satisfying $f(xx) = f(yy)$. Thus $f = c_2$.
The right-hand side is the only non-projection binary function in $\minion{G}$.

Finally such $f$ exist in $\minion{G}$, namely $f = c_2$.
\end{proof}

\begin{lemma}\label{idem} The clone $\clone{I}$ consisting of all idempotent functions has core $\minion{G}$ consisting of all idempotent functions $g$ satisfying $g\leq g^*$.
\end{lemma}
\begin{proof}
A homomorphism from $\clone{I}$ to $\minion{G}$ is provided by the map $g\mapsto g\wedge g^*$, which is a minion homomorphism by Lemmas \ref{binary-minion-homomorphism} and \ref{dual-minion-homomorphism}.

Consider binary functions $f\in\minion{G}$ satisfying $f(xy) = f(yx)$.
Since $f\leq f^*$, we have $f(10)=f(01)=0$. Thus $f(xy) = \wedge$.
The right-hand side is the only non-projection binary function in $\minion{G}$.

Finally such $f$ exist in $\minion{G}$, namely $f=\wedge$.
\end{proof}

\begin{lemma}\label{proj} The clone $\clone{P}$ consisting of all projections is a core.
\end{lemma}
\begin{proof}
By Lemma~\ref{projtoproj}, every minion endomorphism of $\clone{P}$ fixes all projections. Hence every minion endomorphism is the identity, and therefore an automorphism.
\end{proof}

The cores of all products of Boolean clones have now been classified. Figure~\ref{latcores} lists all products of clones together with the lemmas determining their cores.

\begin{figure}
\centering{
{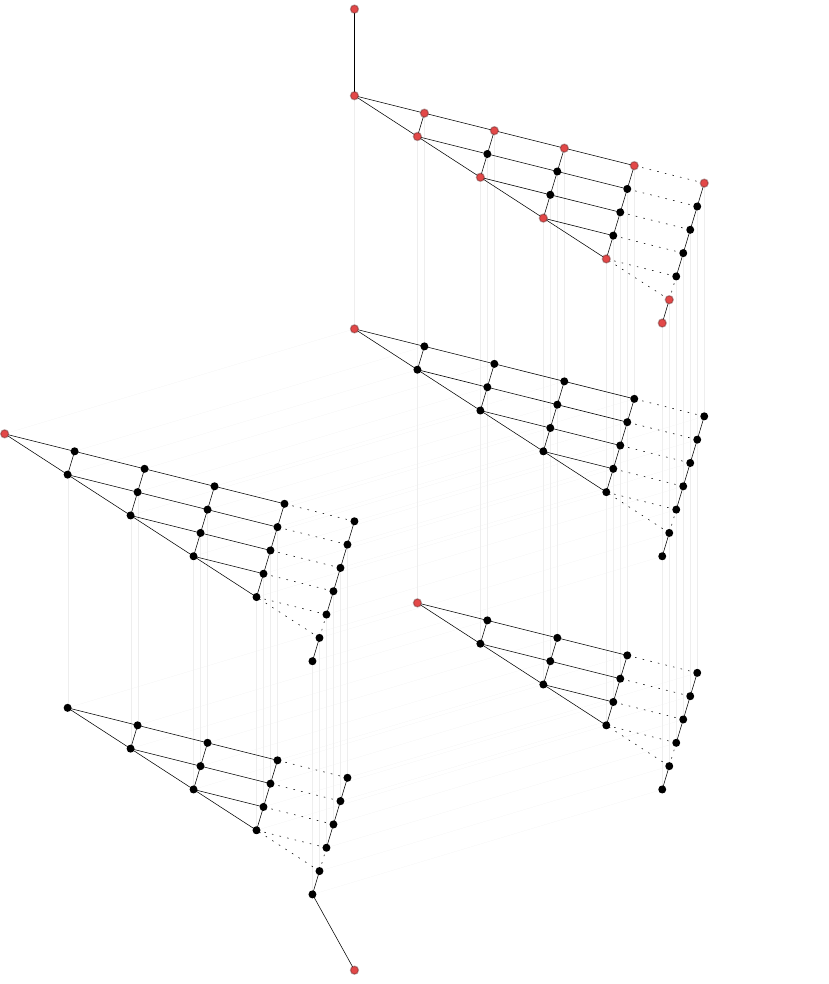}
\caption{Lemmas classifying cores of all products of Boolean clones}
\label{latcores}
}
\end{figure}

\bibliographystyle{plain} 
\bibliography{productClones} 

@article {wonderland,
    AUTHOR = {Barto, Libor and Opršal, Jakub and Pinsker, Michael},
     TITLE = {The wonderland of reflections},
   JOURNAL = {Israel J. Math.},
  FJOURNAL = {Israel Journal of Mathematics},
    VOLUME = {223},
      YEAR = {2018},
    NUMBER = {1},
     PAGES = {363--398},
      ISSN = {0021-2172,1565-8511},
   MRCLASS = {08A40 (03C05 03C15 68Q17 68Q25)},
  MRNUMBER = {3773066},
MRREVIEWER = {Pedro\ S\'anchez Terraf},
       DOI = {10.1007/s11856-017-1621-9},
       URL = {https://doi.org/10.1007/s11856-017-1621-9},
}

@book{post,
 ISBN = {9780691095707},
 URL = {http://www.jstor.org/stable/j.ctt1bgzb1r},
 author = {EMIL L. POST},
 publisher = {Princeton University Press},
 title = {The Two-Valued Iterative Systems of Mathematical Logic. (AM-5)},
 urldate = {2024-07-17},
 year = {1941}
}

@book{bergman,
  title={Universal Algebra: Fundamentals and Selected Topics},
  author={Bergman, C.},
  isbn={9781439851296},
  lccn={2011036515},
  series={Chapman \& Hall Pure and Applied Mathematics},
  url={https://books.google.cz/books?id=QXi3BZWoMRwC},
  year={2011},
  publisher={Taylor \& Francis}
}

@article{bodnarchuk,
  title={Galois theory for post algebras. I},
  author={Victor Bodnarchuk and L. A. Kaluzhnin and V. N. Kotov and Boris A. Romov},
  journal={Cybernetics},
  year={1969},
  volume={5},
  pages={243-252},
  url={https://api.semanticscholar.org/CorpusID:121821029}
}

@article{geiger,
author = {David Geiger},
title = {{Closed systems of functions and predicates.}},
volume = {27},
journal = {Pacific Journal of Mathematics},
number = {1},
publisher = {Pacific Journal of Mathematics, A Non-profit Corporation},
pages = {95 -- 100},
year = {1968},
}

@article {approach,
    AUTHOR = {Barto, Libor and Bulín, Jakub and Krokhin, Andrei and Opršal, Jakub},
     TITLE = {Algebraic approach to promise constraint satisfaction},
   JOURNAL = {J. ACM},
  FJOURNAL = {Journal of the ACM},
    VOLUME = {68},
      YEAR = {2021},
    NUMBER = {4},
     PAGES = {Art. 28, 66},
      ISSN = {0004-5411,1557-735X},
   MRCLASS = {08A70 (68Q17 68Q25)},
  MRNUMBER = {4313115},
MRREVIEWER = {Ágnes\ Szendrei},
       DOI = {10.1145/3457606},
       URL = {https://doi.org/10.1145/3457606},
}

@article{romov,
  title={On the lattice of subalgebras of direct products of post
algebras of finite degree},
  author={Romov, BA},
  journal={Mathematical Models of Complex Systems},
  pages={156--168},
  year={1973},
  publisher={Izd. Inst. Kibernetiki Akad. Nauk UkrSSR Kiev}
}

@article {vucaj,
    AUTHOR = {Bodirsky, Manuel and Vucaj, Albert},
     TITLE = {Two-element structures modulo primitive positive
              constructability},
   JOURNAL = {Algebra Universalis},
  FJOURNAL = {Algebra Universalis},
    VOLUME = {81},
      YEAR = {2020},
    NUMBER = {2},
     PAGES = {Paper No. 20, 17},
      ISSN = {0002-5240,1420-8911},
   MRCLASS = {06E30 (08A70 08B05)},
  MRNUMBER = {4085967},
MRREVIEWER = {Keith\ A.\ Kearnes},
       DOI = {10.1007/s00012-020-0647-8},
       URL = {https://doi.org/10.1007/s00012-020-0647-8},
}

@book{trotter,
  title={Combinatorics and Partially Ordered Sets: Dimension Theory},
  author={Trotter, W.T.},
  isbn={9780801869778},
  lccn={91038443},
  series={Johns Hopkins Studies in Nineteenth C Architecture Series},
  url={https://books.google.cz/books?id=VAEgDWgFZ10C},
  year={1992},
  publisher={Johns Hopkins University Press}
}

@book {universal1,
    AUTHOR = {McKenzie, Ralph N. and McNulty, George F. and Taylor, Walter
              F.},
     TITLE = {Algebras, lattices, varieties. {V}ol. 1},
      NOTE = {Reprint of [MR0883644],
              \copyright1969},
 PUBLISHER = {AMS Chelsea Publishing/American Mathematical Society,
              Providence, RI},
      YEAR = {2018},
     PAGES = {xii+367},
      ISBN = {978-1-4704-4295-8},
   MRCLASS = {08-01 (01A75 06-01)},
  MRNUMBER = {3793673},
MRREVIEWER = {James\ B.\ Nation},
       DOI = {10.1090/chel/383.H},
       URL = {https://doi.org/10.1090/chel/383.H},
}

@book {universal2,
    AUTHOR = {Freese, Ralph S. and McKenzie, Ralph N. and McNulty, George F.
              and Taylor, Walter F.},
     TITLE = {Algebras, lattices, varieties. {V}ol. {II}},
    SERIES = {Mathematical Surveys and Monographs},
    VOLUME = {268},
 PUBLISHER = {American Mathematical Society, Providence, RI},
      YEAR = {2022},
     PAGES = {xvii+475},
      ISBN = {978-1-4704-6797-5; 978-4704-7129-3},
   MRCLASS = {08-02 (03C05 06-02 06Bxx 08Bxx)},
  MRNUMBER = {4496007},
MRREVIEWER = {James\ B.\ Nation},
}

@book {universal3,
    AUTHOR = {Freese, Ralph S. and McKenzie, Ralph N. and McNulty, George F.
              and Taylor, Walter F.},
     TITLE = {Algebras, lattices, varieties. {V}ol. {III}},
    SERIES = {Mathematical Surveys and Monographs},
    VOLUME = {269},
 PUBLISHER = {American Mathematical Society, Providence, RI},
      YEAR = {[2022] \copyright2022},
     PAGES = {xvii+430},
      ISBN = {978-1-4704-6798-2; 978-1-4704-7130-9},
   MRCLASS = {08-02 (03C05 06-02 06Bxx 08Bxx)},
  MRNUMBER = {4563240},
MRREVIEWER = {James\ B.\ Nation},
}

@article{florian1,
author = {Bodirsky, Manuel and Starke, Florian and Vucaj, Albert},
title = {Smooth digraphs modulo primitive positive constructability and cyclic loop conditions},
journal = {International Journal of Algebra and Computation},
volume = {31},
number = {05},
pages = {929-967},
year = {2021},
doi = {10.1142/S0218196721500442},
URL = {https://doi.org/10.1142/S0218196721500442},
eprint = {https://doi.org/10.1142/S0218196721500442}
}

@misc{florian2,
      title={Finite Simple Groups in the Primitive Positive Constructability Poset}, 
      author={Sebastian Meyer and Florian Starke},
      year={2025},
      eprint={2409.06487},
      archivePrefix={arXiv},
      primaryClass={math.GR},
      url={https://arxiv.org/abs/2409.06487}, 
}

@article{florian3,
   title={Maximal Digraphs with Respect to Primitive Positive Constructability},
   volume={42},
   ISSN={1439-6912},
   url={http://dx.doi.org/10.1007/s00493-022-4918-1},
   DOI={10.1007/s00493-022-4918-1},
   number={S1},
   journal={Combinatorica},
   publisher={Springer Science and Business Media LLC},
   author={Bodirsky, Manuel and Starke, Florian},
   year={2022},
   month=oct, pages={997–1010}
}

@article{Taimanov1,
url = {https://doi.org/10.1515/dma-2019-0012},
title = {On some properties of vector functions of Boolean algebra},
author = {Vladimir A. Taimanov},
pages = {131--141},
volume = {29},
number = {2},
journal = {Discrete Mathematics and Applications},
doi = {doi:10.1515/dma-2019-0012},
year = {2019},
lastchecked = {2026-05-25}
}

@article{Taimanov2,
url = {https://doi.org/10.1515/dma-2023-0017},
title = {On bases of all closed classes of Boolean vector functions},
author = {Vladimir A. Taimanov},
pages = {189--198},
volume = {33},
number = {3},
journal = {Discrete Mathematics and Applications},
doi = {doi:10.1515/dma-2023-0017},
year = {2022},
lastchecked = {2026-05-25}
}

@article{Taimanov3,
url = {https://doi.org/10.1515/dma-2020-0024},
title = {On bases of closed classes of Boolean vector functions},
author = {Vladimir A. Taimanov},
pages = {273--283},
volume = {30},
number = {4},
journal = {Discrete Mathematics and Applications},
doi = {doi:10.1515/dma-2020-0024},
year = {2020},
lastchecked = {2026-05-25}
}

@article{kapytka,
author = {Barto, Libor and Kapytka, Maryia},
year = {2024},
month = {11},
pages = {},
title = {Multisorted Boolean clones determined by binary relations up to minion homomorphisms},
volume = {86},
journal = {Algebra universalis},
doi = {10.1007/s00012-024-00878-0}
}

@misc{david,
      title={On the lattice of multi-sorted relational clones on a two-element set}, 
      author={Vojtěch David and Dmitriy Zhuk},
      year={2025},
      eprint={2505.06033},
      archivePrefix={arXiv},
      primaryClass={math.LO},
      url={https://arxiv.org/abs/2505.06033}, 
}

@InCollection{cspsurvey,
  author =	{Barto, Libor and Krokhin, Andrei and Willard, Ross},
  title =	{{Polymorphisms, and How to Use Them}},
  booktitle =	{The Constraint Satisfaction Problem: Complexity and Approximability},
  pages =	{1--44},
  series =	{Dagstuhl Follow-Ups},
  ISBN =	{978-3-95977-003-3},
  ISSN =	{1868-8977},
  year =	{2017},
  volume =	{7},
  editor =	{Krokhin, Andrei and Zivny, Stanislav},
  publisher =	{Schloss Dagstuhl -- Leibniz-Zentrum f{\"u}r Informatik},
  address =	{Dagstuhl, Germany},
  URL =		{https://drops.dagstuhl.de/entities/document/10.4230/DFU.Vol7.15301.1},
  URN =		{urn:nbn:de:0030-drops-69595},
  doi =		{10.4230/DFU.Vol7.15301.1}
}

\end{document}